\documentclass[a4paper,11pt]{amsart}
\usepackage{amssymb,amsmath,accents}
\usepackage{amsfonts,amsthm,mathrsfs}
\usepackage{cases} 
\usepackage{graphicx} 
\usepackage{subcaption}
\usepackage{diagbox}
\usepackage{url}

\usepackage[mathscr]{eucal}
\usepackage{xcolor}
\usepackage[abbrev]{amsrefs}
\usepackage{latexsym}
\usepackage{enumitem}
\usepackage{mathtools}

\usepackage{ulem}
\DeclareRobustCommand{\erase}{\bgroup\markoverwith{\textcolor{blue}{\rule[.5ex]{2pt}{0.4pt}}}\ULon}

\newtheorem{Theorem}{\bf Theorem}[section]
\newtheorem{Proposition}{\bf Proposition}[section]
\newtheorem{Lemma}{\bf Lemma}[section]

\theoremstyle{definition}
\newtheorem{Remark}{\bf Remark}[section]

\newtheorem{Definition}{\bf Definition}[section]

\newcommand{\loc}{{\mathrm{loc}}}
\renewcommand{\div}{\operatorname{div}}

\newcommand{\R}{\mathbb{R}}

\newcommand{\Z}{\mathbb{Z}}

\newcommand{\II}{I\!I}
\newcommand{\III}{I\!I\!I}
\newcommand{\IV}{I\!V}

\makeatletter
\newcommand{\bracesize}[1]{\bBigg@{#1}}
\makeatother
\numberwithin{equation}{section}

\allowdisplaybreaks[4]
\makeatother

\begin{document}

\title[]{Large time behavior of exponential surface diffusion flows on $\mathbb{R}$}

\author[Y.~Giga]{Yoshikazu Giga}
\address[Y.~Giga]{Graduate School of Mathematical Sciences, The University of Tokyo, 3-8-1 Komaba, Meguro-ku, Tokyo 153-8914, Japan.}
\email{labgiga@ms.u-tokyo.ac.jp}

\author[M.~G{\"o}{\ss}wein]{Michael G{\"o}{\ss}wein}
\address[M.~G{\"o}{\ss}wein]{Department of mathematics, University of Duisburg-Essen, Thea-Leymann-Stra{\ss}e 9, 45127, Essen, Germany.}
\email{michael.goesswein@uni-due.de}

\author[S.~Katayama]{Sho Katayama}
\address[S.~Katayama]{Graduate School of Mathematical Sciences, The University of Tokyo, 3-8-1 Komaba, Meguro-ku, Tokyo 153-8914, Japan.}
\email{katayama-sho572@g.ecc.u-tokyo.ac.jp}


\begin{abstract}
We consider a surface diffusion flow of the form $V=\partial_s^2f(-\kappa)$ with a strictly increasing smooth function $f$ typically, $f(r)=e^r$, for a curve with arc-length parameter $s$, where $\kappa$ denotes the curvature and $V$ denotes the normal velocity.
 The conventional surface diffusion flow corresponds to the case when $f(r)=r$.
 We consider this equation for the graph of a function defined on the whole real line $\mathbb{R}$.
 We prove that there exists a unique global-in-time classical solution provided that the first and the second derivatives are bounded and small.
 We further prove that the solution behaves like a solution to a self-similar solution to the equation $V=-f'(0)\kappa$.
 Our result justifies the explanation for grooving modeled by Mullins (1957) directly obtained by Gibbs--Thomson law without linearization of $f$ near $\kappa=0$.
\end{abstract}

\maketitle


\section{Introduction} \label{Sint} 

\subsection{A general surface diffusion flow} \label{SSSD}
We consider a general surface diffusion flow
\[
	{V=\partial_s^2f(-\kappa)}
\]
with a given increasing function $f$, typically $f(r)=e^r$, for a graph-like curve $\Gamma_t=\left\{ y=u(x,t) \right\}$ defined on the whole real line $\mathbb{R}$.
 Here $V$ denotes the upward normal velocity and $\kappa$ denotes the upward curvature;
 $\partial_s$ denotes the derivative with respect to arc-length parameter so that $\partial_s^2$ is the Laplace--Beltrami operator on a curve.
 If $f(r)=r$, this gives us the usual surface diffusion flow
\[
	{V=-\partial_s^2\kappa.}
\]
Since $\partial_s=(1+u_x^2)^{-1/2}\partial_x$, $\kappa=-u_{xx}/(1+u_x^2)^{1/2}$ and $V=u_t/(1+u_x^2)^{1/2}$ with $u_t=\partial u/\partial t$, $u_x=\partial u/\partial x$, $u_{xx}=\partial_x^2u$, the resulting initial value problem for $u$ is of the form
    \begin{equation}\label{SDE'}
  u_t=\left(\frac{1}{(1+u_x^2)^{1/2}}\left(f\left(-\frac{u_{xx}}{(1+u_x^2)^{3/2}}\right)\right)_x\right)_{x}\ \textrm{in}\ \R\times(0,T),\ u(\cdot,0)=u_0\ {\textrm{in}\ \R.}
\end{equation}

Let us first explain a physical background of the equation $V=\partial_s^2f(-\kappa)$.
 W.\ W.\ Mullins \cite{Mu2} modeled a relaxation dynamics of crystal surface such as gold (Au) in low temperature. {(See also e.g., {for derivation of the model} \cites{KM, LLMM}.)}
 Let $\Gamma_t$ denote the crystal surface at time $t$ and $V$ denote its normal velocity.
 The evolution law is 
of the form
\[
	V = -\operatorname{div}_\Gamma j,
\]
where $j$ represents mass flux and $\operatorname{div}_\Gamma$ denotes the surface divergence. 
 We postulate that
\[
	j = - M \nabla_\Gamma \rho,
\]
where $M>0$ is a kinetic coefficient and $\nabla_\Gamma$ denotes the surface gradient;
 $\rho$ denotes the density.
 The Gibbs--Thomson law is of the form
\[
	\log \left( \frac{\rho}{\rho_0} \right) = \beta \frac{\delta E}{\delta\Gamma},
\]
where $\rho_0$, $\beta$ are positive constants;
 $\rho_0$ denotes a density of equilibrium and $\beta=(k\theta)^{-1}$ where $k$ is the Boltzmann constant and $\theta$ denotes the temperature.
 $\delta E/\delta\Gamma$ denotes the variation of surface energy $E$.
 If the surface energy $E$ is an area, $-\delta E/\delta\Gamma$ is the (twice) mean curvature $\kappa$ of $\Gamma$.
 In the case of {the diffusion limited model (\cite{KM},\cite{LLMM})}, where $M$ is assumed to be constant (independent of the normal $\mathbf{n}$ of $\Gamma$), these three relations yield an equation
\[
	V = M \Delta_\Gamma \left( \rho_0 \exp(-\beta\kappa) \right),
\]
where $\Delta_\Gamma=\operatorname{div}_\Gamma \nabla_\Gamma$ denotes the Laplace--Beltrami operator.
 This is an exponential type of surface diffusion flow.
 We may rescale space and time so that $M\rho_0=1$ and that $\beta=1$.
 The resulting equation is
\begin{equation}\label{expSDE}
	V = \Delta_\Gamma f(-\kappa)
\end{equation}
with $f(r)=e^r$. If one linearizes {$f(-\kappa)$} around $\kappa=0$, then {\eqref{expSDE} becomes}
\begin{equation}\label{convSDE}
V=-{\, f'(0)}\Delta_\Gamma\kappa,
\end{equation}
  which is nothing but a conventional surface diffusion flow.
 In one-dimensional setting, this is \erase{nothing but} $V=\partial_s^2f(-\kappa)$.
 Mullins \cite{Mu2} also modeled evolution of crystal surface in evaporation-condensation dynamics.
 A typical example is Magnesium (Mg) in low temperature.
 In this case, the Gibbs--Thomson law is of the form
\[
	\log \left( \frac{p}{p_0} \right) = \beta \frac{\delta E}{\delta\Gamma},
\]
where $p$ denotes the pressure and $p_0$ denotes the atmospheric pressure.
 The proposed evolution law is of the form
\[
	V = m (p_0 - p)
\]
with a positive constant $m$.
 Thus the resulting equation is
\[
	V = mp_0 \left( 1-\exp(-\beta\kappa) \right)
\]
if $\delta E/\delta\Gamma=-\kappa$.
 Linearizing around $\kappa=0$, we obtain $V=mp_0\beta\kappa$, which is nothing but the mean curvature flow \erase{equation} originally proposed by Mullins \cite{Mu1} to model grain boundary motion in the process of annealing.

  The existence and large time behavior of a solution to the (conventional) surface diffusion flow \eqref{convSDE} have been studied in various settings. Among other results, Escher, Mayer, and Simonett \cite{EMS} proved that if the initial surface $\Gamma_0$ is closed and sufficiently close to a sphere $S$, then a global-in time solution $\Gamma_t$ to the equation \eqref{convSDE} uniquely exists and converges to some sphere $S'$, which is a stationary solution to \eqref{convSDE}, at an exponential rate. We note that $S'$ may be different from $S$. Their method is a stability analysis using the Lyapunov–-Schmidt reduction (to deal with the indefiniteness of spheres.) Asymptotic stability of the spheres for the conventional surface diffusion flow has been studied by many researchers, see e.g., \cites{EM,LSS,Whe1,Whe2,Whe3}. In contrast, {for} the exponential surface diffusion flow{, such problems seem to remain open.}
  
It is revealed by several mathematicians that a different type of asymptotic behavior may occur when the solution surface $\Gamma_t$ is unbounded. Koch and Lamm \cite{KL} proved the existence of a unique global-in-time solution $\Gamma_t$ to the equation \eqref{convSDE} for a graph-like surface $\Gamma_t=\{(x,u(x,t));\ x\in\R^n\}$. They proved that under the initial condition $u(\cdot,0)=u_0$ with the quantity $\|\nabla u_0\|_{L^\infty(\R^n)}$ sufficiently small, the equation \eqref{convSDE} possesses a unique smooth global-in-time solution $u$ with the Sobolev-type norm
\begin{equation}\label{KLnorm}
  \|u\|_{X_\infty}:=\|\nabla_x u\|_{L^\infty(\R^n\times\R)}+\sup_{x\in\R^n,R>0}R^\frac{2}{n+6}\|\nabla_x^2u\|_{L^{n+6}(B_R(x)\times(R^4/2,R^4))}
\end{equation}
sufficiently small.
 Here, we remark that {they have chosen the function space $C^{0,1}(\R^N)$, that is the space of all (possibly unbounded) Lipschitz continuous functions, for the initial value $u_0$ and 
$X_\infty$ for solutions $u$, so that they are} invariant under the scaling
  \begin{equation}\label{scaling}
    u\mapsto u^\sigma(x,t):=\sigma^{-1/4} u(\sigma^{1/4}x,\sigma t)
  \end{equation}
for any $\sigma>0${.
 Since the equation \eqref{convSDE} is invariant under this scaling,} the solution surface $\Gamma_t=\{(x,u(x,t));\ x\in\R^n\}$ for the flow \eqref{convSDE} is {scale-invariant provided that the initial data is invariant under this scaling}. This result, in particular, implies that for any Lipschitz function $A$ on $S^{N-1}$ with $\|\nabla_{S^{N-1}} A\|_{L^\infty(S^{N-1})}$ sufficiently small, the equation \eqref{convSDE} possesses a unique smooth \textit{self-similar solution} $u=u_A$, that is a solution which is expressed by 
  \[
  u_A(x,t)=t^{1/4}\Phi_A(t^{-1/4}x)\footnote{This is equivalent to the invariance under the scaling \eqref{scaling} for any $\sigma>0$.},
  \]
  with the condition
  \[
  \lim_{r\to\infty}\sup_{\omega\in S^{N-1}}|r^{-1}\Phi_A(r\omega)-A(\omega)|=0
  \]
  and the norm $\|\nabla u\|_{X_\infty}$ sufficiently small. It is proved by solving the equation \eqref{convSDE} starting with the graph of the function $u_{0A}=|x|{A(x/|x|)}$, which is invariant under the scaling \eqref{scaling}.
 The idea to construct a self-similar solution to an evolution equation by solving the initial value problem with homogeneous initial data goes back to \cite{GM}, where the authors constructed non-trivial self-similar solutions to the vorticity equations which are formally equivalent to the Navier–Stokes equations.
  
  Later on, Du and Yip \cite{DY} obtained the result on a large time behavior of solutions to the flow \eqref{convSDE} of \textit{asymptotically self-similar} type for $u_0$ with a small gradient.
 Among other results, they proved that if $u_0-u_{0A}$ is bounded and the quantities $\|\nabla u_0\|_{L^\infty(\R^n)}$ and $\|\nabla(u_0-u_{0A})\|_{L^\infty(\R^n)}$ are sufficiently small, the solution $u$ constructed in \cite{KL} satisfies
  \[
  \lim_{t\to\infty}\|t^{-1/4}u(t^{1/4}\,\cdot\,,t)-\Phi_A\|_{L^\infty(K)}=0
  \]
for any compact set $K\subset\R^n$.
 This means the profile of the solution $u(x,t)$ tends to that of the self-similar solution $u_A$ as $t\to\infty$. {We also mention \cite{AG}, in which the authors proved the existence and stability of a globally bounded self-similar profile of the equation \eqref{SDE'} in $(0,\infty)\times(0,\infty)$ with a contact angle condition $u_x(0,\cdot)=\beta$ on $(0,\infty)$ and a linearized no-flux condition $u_{xxx}(0,\cdot)=0$ on $(0,\infty)$ with the contact angle $\beta$ sufficiently small.}
 These results on convergence to self-similar solutions are obtained in the spirit of \cite{GGS}, where the vorticity equations are mainly discussed.
 For more results on the existence and large-time behavior of a solution to a surface diffusion flow \eqref{convSDE}, see e.g. {\cites{Asa, AG, EMS, EM, GG1, GG2, GIK, Goess, IK, LSS, LS1, LS2, Whe1, Whe2, Whe3}} and references therein.

  On the other hand, for an exponential surface diffusion flow \eqref{expSDE} and its graphical form \eqref{SDE'}, {the existence and large time behavior of solutions are widely open. In particular, for the problem \eqref{SDE'}, the lack of the invariance under the scaling \eqref{scaling}, due to non-degenerate second and higher order derivatives of the curvature nonlinearity $f$, prevents us from using the norm $\|u\|_{X_\infty}$ in \cites{KL,DY}, and} to our knowledge, there are no known results on the existence of a  solution to the initial value problem \eqref{SDE'}, even in the sense of a local-in-time solution.
  {
  \begin{Remark}
    Hamamuki \cite{Ham} obtained a result on a unique existence and a self-similar type large time behavior of a global-in-time graph-like solution of the exponential mean curvature flow $V=1-\exp(-\kappa)$.
 {Among other results, in \cite{Ham} he proved} that for any graph-like solution $\Gamma_t=\{(x,u(x,t)); x\in\R^N,x_1>0\}$ with a prescribed contact angle condition $u_{x_1}=\beta$ on $\{x_1=0\}$, scaled solutions $u^\sigma(x,t):=\sigma^{-1/2}u(\sigma^{1/2}x,\sigma t)$ converge to some function $U$ the graph of which is a self-similar solution to the mean curvature flow $V=\kappa$, on $\{(x,t)\in\R^N\times(0,\infty); x_1>0\}$ as $\sigma\to\infty$. We point out that their argument is based on viscosity solutions, and does not impose any smallness assumptions unlike \cites{DY,KL}. Unfortunately, since their method relies on the comparison principle, it is not applicable to our equation \eqref{SDE'}, which is of fourth order and hence possesses no comparison property.
  \end{Remark}
  }
  In this paper, {upon a H\"{o}lder space framework (see also Remark~\ref{Holrmk}),} we establish a global-in-time solvability for the problem \eqref{SDE'} with a general curvature nonlinearity $f$ with initial data with a small gradient (see Theorem~\ref{GE}). Furthermore, we obtain a large time behavior of a global-in-time solution to the equation \eqref{SDE'} of asymptotically self-similar type. In particular, we prove that $u^\sigma$ converges to a function $U$ on $\R\times(0,\infty)$ such that the graph of $U(\cdot,t)$ is a solution to the surface diffusion flow $V=-f'(0)\partial_s^2\kappa$ (see Theorem~\ref{ConvSSS}). We also mention a local-in-time solvability for large initial data and an unconditional uniqueness {on H\"{o}lder spaces} (see Theorem~\ref{locexistence}).

	In the proof of the global existence and the large time behavior, we adopt scaled H\"{o}lder norms $\|\cdot\|_{BC^{k+\mu,(k+\mu)/4}(\R\times(a,b))}'$ (see Definition~\ref{defHol}~(3)) and 
properly weighted H\"{o}lder norms $\|\cdot\|_{Z^k_T}$ (see Definition~\ref{defZkt} and Remarks~\ref{rmkZkt},~\ref{rmkZkt2}), instead of usual parabolic H\"{o}lder norms $\|\cdot\|_{BC^{k+\mu,(k+\mu)/4}(\R\times(a,b))}$. This technical improvement enables us to establish linear and nonlinear estimates without breaking the (asymptotic) scale invariance, and makes them simple and unified.
  \begin{Remark}\label{expect}
  We explain the reason why we expect that $u^\sigma$ converges to $U$. Let $\Gamma_t$, $V(\cdot,t)$ and $\kappa(\cdot,t)$ be a graph of the function $u(\cdot,t)$, the normal velocity of $\Gamma_t$, and the curvature of $\Gamma_t$, respectively. Also, define $\Gamma^\sigma_t$, $V^\sigma_t$ and $\kappa^\sigma(\cdot,t)$ in the same manner. Then $\Gamma^\sigma_t=\sigma^{-1/4}\Gamma_{\sigma t}$, and
  \[
  V^\sigma(z,t)=\sigma^{3/4}V(\sigma^{1/4}z,\sigma t),\ \kappa^\sigma(z,t)=\sigma^{1/4}\kappa(\sigma^{1/4}z,\sigma t).
  \]
  In particular, if $\Gamma_t$ solves the equation $V=\partial_s^2 f(-\kappa)$, then
  \[
  V^\sigma(z,t)=\sigma^{1/4}\partial_s^2 f(-\kappa(\sigma^{1/4}z,\sigma t))=\partial_s^2 f_\sigma(-\kappa^\sigma),
  \]
  where $f_\sigma(r):=\sigma^{1/4}f(-\sigma^{-1/4}r)$ (see also \eqref{scaledeq}). Since $f_\sigma(r)\to f'(0)$ as $\sigma\to\infty$, we expect that $\Gamma^\sigma_t$ converges to a solution to $V=-f'(0)\partial_s^2\kappa$.
  \end{Remark}
 \subsection{Main results} \label{SSMain}

Our main results are summarized as follows.
 We shall work in \erase{a} parabolic H{\"o}lder spaces.
 For detailed notation, see Section \ref{SPr}.
 The first result is the unique existence of a global-in-time solution with decay estimates when $\lVert u_{0x}\rVert_\infty$, $\lVert u_{0xx}\rVert_\infty$ are small.
 The second result says that the solution asymptotically becomes a self-similar solution to the surface diffusion flow $V=-f'(0)\partial_s^2\kappa$.
 Let $W^{k,p}$ denote the $L^p$-Sobolev space of order $k$ for $k\in\Z_{\ge 0}$ and $p\in[1,\infty]$.
 As well-known, $W^{1,\infty}=C^{0,1}$ (see e.g., \cite[5.8.2, Theorem 4]{E}). See also Section~2 for notation of function spaces and norms.

\begin{Theorem}\label{GE}
    Let $k\in\Z_{\ge 4}$ {(i.e., $k$ is an integer and $k\ge4$),} $\mu\in(0,1)$ and suppose that
       \begin{gather}
    f\in W^{k,\infty}_{\textrm{loc}}(\R),\ f'>0,\ f'(0)=1,\label{fassump}\\
    u_0\in C^{4+\mu}(\R),\ (u_0)_x\in L^\infty(\R).\label{u0assmup}
  \end{gather}            
 Then there is a constant $\varepsilon_0>0$ such that if $\|(u_0)_x\|_{W^{1,\infty}(\R)}<\varepsilon_0$, then problem \eqref{SDE'} possesses a unique global-in-time solution $u\in u_0+BC^{4+\mu,1+\mu/4}(\R\times(0,\infty))$. Furthermore, we have that $u\in C^{k+1+\mu,(k+1+\mu)/4}_{\textrm{time loc}}(\R\times(0,\infty))$ and that
    \begin{equation}\label{Thmudecay}
    \begin{multlined}
\|u_x\|_{L^\infty(\R\times(0,\infty))}+\sup_{t\in(0,\infty)}\sum_{\substack{\ell,m\in\Z_{\ge 0},\\ 1\le\ell+4m\le k+1}}t^{(\ell+4m-1)/4}(1+t)^{1/4}\|\partial_x^\ell\partial_t^m u\|_{L^\infty(\R\times(0,\infty))}\\
      +\sup_{t\in(0,\infty)}t^{(k+\mu)/4}(1+t)^{1/4}[u]_{C^{k+1+\mu,(k+1+\mu)/4}(\R\times(t/2,t))}\le C\|(u_0)_x\|_{W^{1,\infty}},
    \end{multlined}
    \end{equation}
    with $C$ depending only on $f$, $\varepsilon_0$ and $k$.
  \end{Theorem}
\begin{Theorem}\label{ConvSSS}
  Assume the same conditions as in Theorem~\ref{GE}. There is a constant $\varepsilon_*>0$ such that if $\|(u_0)_x\|_{W^{1,\infty}(\R)}<\varepsilon_*$ and
  \begin{equation}\label{IVasympt}
  u_0(x)=(a+o(1))x\textrm{ as }x\to+\infty,\ (b+o(1))x\textrm{ as }x\to-\infty,
  \end{equation}
  for some $a,b\in(-\varepsilon_*,\varepsilon_*)$, then $u^\sigma(x,t):=\sigma^{-1}u(\sigma x,\sigma^4t)$ satisfies
  \begin{equation}\label{asymptSSS}
    u^\sigma\to U\ \textrm{uniformly on every compact set in $\R\times[0,\infty)$}\ \textrm{as}\ \sigma\to\infty,
  \end{equation}
 where $U\in C^\infty(\R\times(0,\infty))\cap C(\R\times[0,\infty))$ is the unique self-similar solution to the problem
  \begin{equation}\label{SSS}
    U_t=\left(\frac{1}{(1+U_x^2)^{1/2}}\left(-\frac{U_{xx}}{(1+U_x^2)^{3/2}}\right)_x\right)_x\ \textrm{in}\ \R\times(0,\infty),\ U(x,0)=\begin{cases}
      ax& \textrm{ if $x\ge 0$,}\\
      bx& \textrm{ if $x<0$.}
    \end{cases}
  \end{equation}
  Furthermore, $\partial_x^{\ell}\partial_t^m u^\sigma$ converges to $\partial_x^{\ell}\partial_t^m U$ uniformly on every compact set in $\R\times(0,\infty)$ as $\sigma\to\infty$ for all $\ell,m\in\Z_{\ge 0}$ with $\ell+4m\le k+1$.
\end{Theorem}

The key idea for the proofs of Theorems~\ref{GE}~and~\ref{ConvSSS} is to obtain the decay estimate \eqref{Thmudecay}. Once this estimate is obtained, we can extend a local-in-time solution (see Theorem~\ref{locexistence} for the existence) to a global-in-time solution. Furthermore, the use of the {Ascoli-Arzel\`a} theorem enables us to justify a formal limit argument in Remark~\ref{expect} (see also \eqref{scaledeq}).

In order to obtain the estimate \eqref{Thmudecay}, we differentiate the both sides of the equation \eqref{SDE'} and re-write it in the form
\begin{equation}\label{introlineq}
  v_t=((1-\alpha)v_{xx}+F)_{xx}\quad \textrm{in}\quad \R\times(0,T),\quad v(\cdot,0)=v_0\quad \textrm{in}\quad \R,
\end{equation}
where $v=u_x$ and $\alpha=\alpha(v,v_x)$, $F=F(v,v_x)$ (see \eqref{alpha} and \eqref{F}). Here, the functions $\alpha$ and $F$ are treated as perturbations from the biharmonic heat equation and expected to be small. To deal with the fourth-order perturbation $\alpha$, we prepare a suitable weighted H\"{o}lder norm (denoted by $\|v\|_{Z^k_T}$) and establish a global-in-time linear estimate by a perturbation argument using semigroup estimates for the biharmonic {heat} semigroup $e^{-t\partial_x^4}$ and the parabolic Schauder estimates (see Section~4).
 We also apply direct estimates to confirm the smallness of $\alpha$ and control the lower-order perturbation $F$ (see Section~5).
\begin{Remark}\label{Holrmk}
  We explain the reason why we adopted a H\"{o}lder space framework. When we try to obtain a nonlinear estimate for a fourth-order \textit{quasilinear} parabolic equation \eqref{introlineq} via a perturbation argument from the biharmonic heat equation, (unlike a semilinear equation), we have to estimate fourth-order {spatial} derivatives of a solution to absorb the fourth-order perturbation term, namely
  \[
  \partial_x^4\int_0^t e^{-(t-s)\partial_x^4}[\alpha v_{xxxx}](s)\,ds.
  \]
  If we try to do this using only semigroup estimates, the coefficient $(t-s)^{-1}$ comes from the fourth-order spatial derivatives, and the estimate results in a failure. In order to avoid this difficulty, we have to use a space-time norm $\|f\|^*$ which satisfies the estimate
  \[
  \left\|\partial_x^4\int_0^t e^{-(t-s)\partial_x^4}f(s)\,ds\right\|^*\le C\|f\|^*.
  \]
  It is well-known that this estimate holds for parabolic H\"{o}lder norms and Sobolev norms, whose corresponding estimates are called Schauder estimates and Calder\'{o}n--Zygmund estimates, respectively. While Koch--Lamm and Du--Yip \cites{DY,KL} used a Sobolev space framework (see also \eqref{KLnorm}) to obtain a perturbation estimate, we adopt a H\"{o}lder space framework, which enables us to control the derivatives of a solution in pointwise sense to deal with a general curvature nonlinearity.
\end{Remark}

\begin{Remark}\label{multidim}
  Our method is applicable to graph-like solutions to a multi-dimentional exponential surface diffusion flow, namely
  \begin{equation}\label{mSDE}
  \begin{gathered}
  u_t=\div\left(\sqrt{1+|\nabla u|^2}P[\nabla u]\nabla f\left(-\frac{P[\nabla u]}{\sqrt{1+|\nabla u|^2}}\nabla^2 u\right)\right)\quad\textrm{in}\quad \R^{n}\times(0,\infty),\\
  u(\cdot,0)=u_0\quad\textrm{in}\quad\R^{n},
  \end{gathered}
  \end{equation}
  where
  \[
  P[\nabla u]:=I-\frac{\nabla u\otimes\nabla u}{1+|\nabla u|^2},
  \]
  and we can obtain results on the global-in-time solvability of Cauchy problems for the equation \eqref{mSDE} and the asymptotic behavior of self-similar type for solutions to the equation \eqref{mSDE}, which are similar to Theorems~\ref{GE}~and~\ref{ConvSSS}. In this paper, we restrict our arguments to the case $N=1$ for simplicity.
\end{Remark}
\begin{Remark} \label{RDuYip}
For the conventional surface diffusion flow, Du and Yip \cite{DY} proved under suitable smallness and spatial decay assumptions on initial data that the rescaled function $u^\sigma$ converges to $U$ uniformly in $\mathbb{R}\times[\eta,1/\eta]$ for any $\eta>0$ not only in $K\times[\eta,1/\eta]$ where $K$ is a compact set in $\mathbb{R}$. 
 Here, $U$ denotes the self-similar solution in Theorem \ref{ConvSSS}.
 They discussed the multi-dimensional case.
 Our convergence result to self-similar solutions does not include this kind of result of Du and Yip \cite{DY} when $f(r)=cr$, $c>0$ i.e., the equation is the conventional surface diffusion flow, since the uniform convergence is on $K$ not in $\mathbb{R}$.
 To prove the uniformity in $\mathbb{R}$, it suffices to prove equi-decay properties for $u^\sigma$ as indicated in \cite{DY} which has been already noted in \cite{GGS}.
 We believe such analysis should work in our setting but we won't carry it out in this paper.
\end{Remark}
The rest of this paper is organized as follows. In Section~2, we define parabolic H\"{o}lder norms and summarize elementary and parabolic estimates for these norms. In Section~3, we prove an existence and uniqueness of a local-in-time solution in a parabolic H\"{o}lder space with the initial value in a (spatial) H\"{o}lder space. In Section~4, which is a key part of the proof of main theorems, we define weighted H\"{o}lder norms $\|v\|_{Z^k_T}$ and obtain global-in-time linear estimates for a perturbed biharmonic heat equation \eqref{introlineq} with this norm. In Section~5,  we obtain elementary estimates of $\alpha(v,v_x)$ and $F(v,v_x)$ by the norm $\|v\|_{Z^k_T}$ to apply the linear estimates in Section~4 to obtain nonlinear decay estimates \eqref{Thmudecay}. This enables us to extend a local-in-time solution constructed in Section~3 to a global-in-time solution, and complete the proof of Theorem~\ref{GE}. In Section~6, we justify the limit argument in Remark~\ref{expect} and complete the proof of Theorem~\ref{ConvSSS}.
  \section{Preliminaries} \label{SPr} 
  In this section, as preliminaries, we give definitions of parabolic H\"{o}lder norms, and state some elementary and parabolic H\"{o}lder estimates which plays essential roles throughout the proofs.
  \subsection{Parabolic H\"{o}lder spaces}
  Our arguments are performed on a H\"{o}lder space framework. In this subsection, we {recall} the definitions and properties of parabolic H\"{o}lder seminorms, norms, and spaces {mainly to fix notation}.
    \begin{Definition}
    Let $f\in C(\R)$.
    \begin{enumerate}
      \item For $\lambda\in[0,\infty)$, we define a H\"{o}lder seminorm $[f]_{C^\lambda(\R)}$ by
      \[
      [f]_{C^\lambda(\R)}:=\sup_{\substack{x,y\in\R,\\x\neq y}}\frac{|f^{(\lfloor\lambda\rfloor)}(x)-f^{(\lfloor\lambda\rfloor)}(y)|}{|x-y|^\lambda}
      \]
      if $\lambda\notin\Z$, and
      \[
[f]_{C^\lambda(\R)}:=\|f^{(\lambda)}\|_{L^\infty(\R)}
      \]
      if $\lambda\in\Z$. We also define a H\"{o}lder space $BC^\lambda(\R)$ as the space which consists of function $f\in C(\R)$ such that $\|f\|_{BC^\lambda(\R)}<\infty$.
      \item For $\lambda\ge 0$, we define a H\"{o}lder norm $\|f\|_{BC^\lambda(\R)}$ by
      \[
      \|f\|_{BC^\lambda(\R)}:=\sum_{\ell=0}^{\lfloor\lambda\rfloor}\|f^{(\ell)}\|_{L^\infty(\R)}+[f]_{C^{\lambda}(\R)}.
      \]
      We also define a H\"{o}lder space $BC^\lambda(\R)$ as the space which consists of function $f\in C(\R)$ such that $\|f\|_{BC^\lambda(\R)}<\infty$.
    \end{enumerate}
    \end{Definition}
    \begin{Definition}\label{defHol}
      Let $-\infty\le a<b\le\infty$ and $f\in C(\R\times(a,b))$.
      \begin{enumerate}
        \item For $\lambda,\mu\in(0,1)$, we define H\"{o}lder seminorms $[f]_{C^\lambda_x(\R\times(a,b))}$ and $[f]_{C^\mu_t(\R\times(a,b))}$ by
      \[
[f]_{C^\lambda_x(\R\times(a,b))}:=\sup_{\substack{x,y\in\R,\\t\in(a,b),\\ x\neq y}}\frac{|f(x,t)-f(y,t)|}{|x-y|^\lambda},\
        [f]_{C^\mu_t(\R\times(a,b))}:=\sup_{\substack{x\in\R,\\s,t\in(a,b),\\ s\neq t}}\frac{|f(x,s)-f(x,t)|}{|s-t|^\mu}.
      \]
      We also define
      \[
      [f]_{C^0_{x}(\R\times(a,b))}=[f]_{C^0_t(\R\times(a,b))}:=\|f\|_{L^\infty(\R\times(a,b))}.  
      \]
      \item For $\lambda\ge 0$ and $n\in\Z_{\ge 1}$, we define a H\"{o}lder seminorm $[f]_{C^{\lambda,\lambda/n}(\R\times(a,b))}$ by
      \[
      [f]_{C^{\lambda,\lambda/n}(\R\times(a,b))}:=\sum_{\substack{\ell,m\in\Z_{\ge 0},\\ \ell+nm=\lfloor\lambda\rfloor}}[\partial^\ell_x\partial^m_t f]_{C^{\lambda-\lfloor\lambda\rfloor}_x(\R\times(a,b))}+\sum_{\substack{\ell,m\in\Z_{\ge 0},\\ \lambda-n<\ell+nm\le \lambda}}[\partial_x^\ell\partial_t^mf]_{C^{(\lambda-\ell-nm)/n}_t(\R\times(a,b))}.
      \]
        We also define a H\"{o}lder space $C^{\lambda,\lambda/n}(\R\times(a,b))$ as the space which consists of functions $f\in C(\R\times(a,b))$ such that $[f]_{C^{\lambda,\lambda/n}(\R\times(a,b))}<\infty$.
      \item For $\lambda\ge 0$ and $n\in\Z_{\ge 1}$, we define a H\"{o}lder norm $\|f\|_{BC^{\lambda,\lambda/n}(\R\times(a,b))}$ by
      \[
    \|f\|_{BC^{\lambda,\lambda/n}(\R\times(a,b))}:=\sum_{\substack{\ell,m\in\Z_{\ge 0},\\ \ell+nm\le\lambda}}\|\partial^\ell_x\partial^m_t f\|_{L^\infty(\R\times(a,b))}+[f]_{C^{\lambda,\lambda/n}(\R\times(a,b))}.
      \]
      In the case $(a,b)$ is bounded, we define a scaled H\"{o}lder norm $\|f\|_{BC^{\lambda,\lambda/n}(\R\times(a,b))}'$ by
      \[
        \|f\|_{BC^{\lambda,\lambda/n}(\R\times(a,b))}':=\sum_{\substack{\ell,m\in\Z_{\ge 0},\\ \ell+nm\le\lambda}}(b-a)^{\ell/n+m}\|\partial^\ell_x\partial^m_t f\|_{L^\infty(\R\times(a,b))}+(b-a)^{\lambda/n}[f]_{C^{\lambda,\lambda/n}(\R\times(a,b))}.
      \]
      We also define a H\"{o}lder space $BC^{\lambda,\lambda/n}(\R\times(a,b))$ as the space which consists of function $f\in C(\R\times(a,b))$ such that $\|f\|_{BC^{\lambda,\lambda/n}(\R\times(a,b))}<\infty$. 
      \item For $\lambda>0$ and $n\in\Z_{\ge 1}$, we define a time-local H\"{o}lder space $C^{\lambda,\lambda/n}_{\textrm{time loc}}(\R\times(a,b))$ as the space which consists of functions $f\in C(\R\times(a,b))$ such that $f\in C^{\lambda,\lambda/n}(\R\times(c,d))$ for all $a<c<d<b$.
      \end{enumerate}
    \end{Definition}
  \begin{Remark}\label{Holrmk}
    \begin{enumerate}
    We state some frequently used properties of H\"{o}lder norms, proofs of which are elementary and omitted.
      \item For $f\in BC^{\lambda,\lambda/n}(\R\times(a,b))$ and $\ell,m\in\Z_{\ge 0}$ with $\ell+nm\le\lambda$,
      \begin{gather*}
      \|\partial_x^\ell\partial_t^mf\|_{BC^{\lambda'-\ell-nm,(\lambda'-\ell-nm)/n}(\R\times(a,b))}\le\|f\|_{BC^{\lambda,\lambda/n}(\R\times(a,b))},\\
      \|\partial_x^\ell\partial_t^mf\|_{BC^{\lambda'-\ell-nm,(\lambda'-\ell-nm)/n}(\R\times(a,b))}'\le t^{-(\ell/n+m)}\|f\|_{BC^{\lambda,\lambda/n}(\R\times(a,b))}'.
      \end{gather*}
      \item If $-\infty<a<b<\infty$, for $f\in BC^{\lambda,\lambda/n}(\R\times(a,b))$ and $\lambda'\le\lambda$,
      \[
      \|f\|_{BC^{\lambda',\lambda'/n}(\R\times(a,b))}\le\|f\|_{BC^{\lambda,\lambda/n}(\R\times(a,b))},\ \|f\|_{BC^{\lambda',\lambda'/n}(\R\times(a,b))}'\le\|f\|_{BC^{\lambda,\lambda/n}(\R\times(a,b))}'.
      \]
      \item Scaled H\"{o}lder norm $\|f\|_{BC^{\lambda,\lambda/n}(\R\times(a,b))}'$ possesses a scale invariance property, i.e., if we define $f^\sigma(x,t):=f(\sigma x,\sigma^nt)$ for $\sigma>0$, then
      \[
      \|f^\sigma\|_{BC^{\lambda,\lambda/n}(\R\times(\sigma^{-n}a,\sigma^{-n}b))}'=\|f\|_{BC^{\lambda,\lambda/n}(\R\times(a,b))}'.
      \]
    \end{enumerate}
  \end{Remark}
    We note here that the scaled H\"{o}lder norms $\|f\|_{BC^{\lambda,\lambda/n}(\R\times(a,b))}'$ and their properties stated below will play an essential role in linear and nonlinear \textit{decay estimates for derivatives} for solutions to parabolic equations, which will be discussed in Sections~4~and~5. 
  \begin{Proposition}\label{CompProd}
     Let $-\infty\le a<b\le\infty$, $\lambda>0$, $R>0$, and $g_1,\ldots,g_k,h_1,\ldots,h_k\in BC^{\lambda,\lambda/n}(\R\times(a,b))$ be such that $\|g_i\|_{BC^{\lambda,\lambda/n}(\R\times(a,b))},\|h_i\|_{BC^{\lambda,\lambda/n}(\R\times(a,b))}\le R$ for $i\in\{1,\ldots k\}$.
    \begin{enumerate}[label=(\roman*)]
      \item 
      \begin{gather}
      \|g_1\ldots g_k\|_{BC^{\lambda,\lambda/n}(\R\times(a,b))}\le C_{\lambda,n,k}\prod_{i=1}^k\|g_i\|_{BC^{\lambda,\lambda/n}(\R\times(a,b))},\label{prod}\\
      \begin{aligned}
      \|g_1\ldots g_k-h_1\ldots &h_k\|_{BC^{\lambda,\lambda/n}(\R\times(a,b))}\\
      &\le C_{\lambda,n,k}R^{k-1}\max_{i=1,\ldots,k}\|g_i-h_i\|_{BC^{\lambda,\lambda/n}(\R\times(a,b))}.\end{aligned}\label{dprod}
      \end{gather}
      \item If $A\in W^{\lceil\lambda\rceil,\infty}(\R^k)$, then
      \begin{equation}\label{comp}
      \begin{aligned}
    \|A(g_1,\ldots,&g_k)\|_{BC^{\lambda,\lambda/n}(\R\times(a,b))}\\
    &\le |A(0,\ldots,0)|+C_{\lambda,n,k,R,A}\max_{i\in\{1,\ldots,k\}}\|g_i\|_{BC^{\lambda,\lambda/n}(\R\times(a,b))}.
      \end{aligned}
    \end{equation}
    Furthermore, if $A\in W^{\lceil\lambda\rceil+1,\infty}(\R^k)$, then
       \begin{equation}\label{dcomp}
       \begin{aligned}
      \|A(g_1,\ldots,g_k)-&A(h_1,\ldots,h_k)\|_{BC^{\lambda,\lambda/n}(\R\times(a,b))}\\
      &\le C_{\lambda,n,k,R,A}\max_{i=1,\ldots,k}\|g_i-h_i\|_{BC^{\lambda,\lambda/n}(\R\times(a,b))}.
      \end{aligned}
    \end{equation}
    \end{enumerate}
If $-\infty<a<b<\infty$, the same assertions hold for scaled norms $\|\cdot\|_{BC^{\lambda,\lambda/n}(\R\times(a,b))}'$ with $C$ {independent of} $b-a$. 
  \end{Proposition}
  \begin{proof}
    The proof of assertion \eqref{prod} is similar to \cite{Goess}*{Lemma 2.16}. Note that it is assumed there that the space domain is bounded and that $k=2$; however, the proof is still applicable to get \eqref{prod}.

    Assertion \eqref{dprod} follows from \eqref{prod} and the fact that
    \[
    g_1\ldots g_k-h_1\ldots h_k=\sum_{j=0}^k(g_j-h_j)\prod_{i=0}^{j-1}g_i\prod_{i=j+1}^k h_i.
    \]
    Assertion \eqref{comp} follows from \eqref{prod}, the chain rule, and the obvious facts that
    \begin{gather*}
      [B(g_1,\ldots,g_k)-B(0,\ldots,0)]_{C^{\lambda}_x}\le C\left(B,\max_{i\in\{1,\ldots,k\}}\|g_i\|_{L^\infty}\right)\max_{i\in\{1,\ldots, k\}}[g_i]_{C^\lambda_x},\\
      [B(g_1,\ldots,g_k)-B(0,\ldots,0)]_{C^{\mu}_t}\le C\left(B,\max_{i\in\{1,\ldots,k\}}\|g_i\|_{L^\infty}\right)\max_{i\in\{1,\ldots, k\}}[g_i]_{C^\mu_t},
    \end{gather*}
    for $B\in W^{1,\infty}_{\loc}(\R^k)$ and $\lambda,\mu\in(0,1]$. Assertion \eqref{dcomp} follows from \eqref{prod}, \eqref{comp}, and the fact that
    \[
    A(g_1,\ldots,g_k)-A(h_1,\ldots,h_k)=\int_0^1 \sum_{j=1}^k \partial_jA((1-\sigma)g_1+\sigma h_1,\ldots,(1-\sigma)g_k+\sigma h_k)(g_j-h_j)\,d\sigma.
    \]
    (Note that all differentiations that appear in $\|\, \cdot\, \|_{BC^{\lambda,\lambda/n}}$ commute over the integral.)

    Assertions for scaled norms follow by their scale invariance.
  \end{proof}
  \begin{Proposition}[Contractivity property of lower order parabolic H\"{o}lder norms]\label{contraction}
    Let $-\infty<a<b<\infty$, $n\in\Z_{\ge 1}$, $k\in\{1,\ldots,n\}$, $\mu\in(0,1)$, and $f\in C^{k,k/n}(\R\times(a,b))$. Then {there is a constant $\beta=\beta_{k,k',\mu}>0$ such that}
    \[
    \|f\|_{BC^{k'+\mu,(k'+\mu)/n}(\R\times(a,b))}\le \|f(\cdot,a)\|_{BC^{k'+\mu}(\R)}+C_{k,k',\mu}(b-a)^{\beta\erase{(k,k')}}\|f\|_{BC^{k+\mu,(k+\mu)/n}(\R\times(a,b))}
    \]
    for all $k'\in\{0,\ldots,k-1\}$.
    In praticular, if $f(\cdot,a)\equiv 0$ in $\R$, then
    \[
    \|f\|_{BC^{k'+\mu,(k'+\mu)/n}(\R\times(a,b))}\le C(b-a)^\beta\|f\|_{BC^{k+\mu,(k+\mu)/n}(\R\times(a,b))}.
    \]
  \end{Proposition}
  \begin{proof}
    The proof can be found in \cite{Goess}*{Lemma 2.17}.
  \end{proof}
  \begin{Remark}
    The contractivity property of lower order property implies that the quantities
    \[
    \|f(\cdot,t)\|_{BC^{k'+\mu}(\R)},\ \|f\|_{BC^{k'+\mu,(k'+\mu)/n}(\R\times(s,t))},\ \|f\|_{BC^{k'+\mu,(k'+\mu)/n}(\R\times(s,t))}'
    \]
    are uniformly (H\"{o}lder) continuous in $a<s<t<b$, for $f\in C^{k,k/n}(\R\times(a,b))$.
  \end{Remark}
  \subsection{Biharmonic heat equations}
  In this subsection, we mention some properties of solutions to the biharmonic heat equation
  \begin{equation}\label{BHE}
    w_t=-\partial_x^4 w+g\quad\textrm{in}\quad\R\times(0,T),\ w(\cdot,0)=w_0\quad\textrm{in}\quad\R.
  \end{equation}
  We define the \textit{biharmonic heat kernel} $b=b(x,t)$ by
  \[
  b(x,t):=\left[\mathcal{F}^{-1}_\xi e^{-t\xi^4}\right](x)=t^{-1/4}\overline{b}(t^{-1/4}x),\ \overline{b}(y):=\frac{1}{2\pi}\int_{\R} e^{-|\xi|^4+\sqrt{-1}y\xi}\,d\xi.
  \]
  We denote by $e^{-t\partial_x^4}h$ the function
  \[
  [e^{-t\partial_x^4}h](x):=\int_{\R}b(x-y,t)h(y)\,dy.
  \]
  Then for all $g\in BC^{0,0}(\R\times[0,T))$ and $w_0\in BC^4(\R)$, the solution $w\in BC^{4,1}(\R\times[0,T))$ to \eqref{BHE}, if exists, has the expression
  \[
  w(t)=e^{-t\partial_x^4}w_0+\int_0^t e^{-(t-s)\partial_x^4}g(s)\,ds.
  \]
  We also have a pointwise estimate
  \begin{equation}\label{besti}
    |\partial_x^\ell\partial_t^m b(x,t)|\le C_{\ell,m}t^{-(\ell+4m+1)/4}\exp\left(-\omega \frac{|x|^{4/3}}{t^{1/3}}\right)
  \end{equation}
  for $\ell,m\in\Z_{\ge 0}$, where $\omega>0$ is a constant. (See \cite{Sol}*{\S 4}.)
  \begin{Proposition}[The smoothing effect of the biharmonic heat equation]\label{Smoothing}
    Suppose that $w_0\in BC^0(\R)$. Then $w(t):=e^{-t\partial_x^4}w_0$ belongs to $C^\infty(\R\times(0,\infty))\cap C(\R\times[0,\infty))$ and solves \eqref{BHE} with $g\equiv 0$. Furthermore, $w$ satisfies
    \begin{equation}\label{smo1}
    \sup_{t>0}\|w\|_{BC^{\lambda,\lambda/4}(\R\times(t/2,t))}'\le C_{\lambda}\|w_0\|_{L^\infty(\R)}
    \end{equation}
    for all $\lambda\ge 0$. If $w_0\in W^{1,\infty}(\R)$, then $w$ also satisfies
    \begin{equation}\label{smo2}
    \sup_{t>0}(1+t)^{1/4}\|w_x\|_{BC^{\lambda,\lambda/4}(\R\times(t/2,t))}'\le C_{\lambda}\|w_0\|_{W^{1,\infty}(\R)}
    \end{equation}
    for all $\lambda\ge 1$.
  \end{Proposition}
  \begin{proof}
    It follows from \eqref{besti} that
    \[
    \|\partial_x^\ell\partial_t^m w(\cdot,t)\|_{L^\infty(\R)}\le Ct^{-\ell/4-m}\|w_0\|_{L^\infty(\R)}
    \]
    for all $\ell,m\in\Z_{\ge 0}$, which implies that
    \begin{equation}\label{smo3}
      \|\partial_x^\ell\partial_t^m w\|_{L^\infty(\R\times(t/2,t))}\le Ct^{-\ell/4-m}\|w_0\|_{L^\infty(\R)}.
    \end{equation}
    Furthermore, interpolation arguments yield that
    \begin{equation}\label{smo4}
      \begin{gathered}
        [\partial_x^\ell\partial_t^mw]_{C^\lambda_x(\R\times(t/2,t))}\le Ct^{-(\ell+\lambda)/4-m}\|w_0\|_{L^\infty(\R)},\\
       [\partial_x^\ell\partial_t^mw]_{C^\mu_t(\R\times(t/2,t))}\le Ct^{-\ell/4-(m+\mu)}\|w_0\|_{L^\infty(\R)}, 
      \end{gathered}
    \end{equation}
     for all $\ell,m\in\Z_{\ge 0}$ and $\lambda,\mu\in[0,1]$. By the definition of scaled H\"{o}lder norms, \eqref{smo3} and \eqref{smo4} imply \eqref{smo1}.
      
      Also, since $w_x(t)=e^{-t\partial_x^4}(w_0)_x$, it follows from \eqref{smo1} that
      \[
      \|w_x\|_{BC^{\lambda,\lambda/4}(\R\times(t/2,t))}^{\prime}\le C\|(w_0)_x\|_{L^\infty(\R)}.
      \]
      On the other hand, \eqref{smo1} directly implies that
      \[
      \|w_x\|_{BC^{\lambda,\lambda/4}(\R\times(t/2,t))}^{\prime}\le Ct^{-1/4}\|w_0\|_{L^\infty(\R)}.
      \]
      Combining these estimates, we deduce \eqref{smo2}.
  \end{proof}

  \subsection{Parabolic Schauder estimates}
  In this subsection, we mention the parabolic Schauder estimates for the Cauchy problem
  \begin{equation}\label{LCP}
    w_t=\sum_{i=0}^4 a_i(x,t)\partial_x^i w+g\ \textrm{in}\ \R\times(0,T),\ w(\cdot,0)=w_0\ \textrm{in}\ \R.
  \end{equation}
  \begin{Proposition}[Parabolic Schauder estimate, \cite{Sol}*{Theorem 4.10 in Chapter \IV}]\label{Schauder}
    Suppose that $\lambda\in(0,\infty)\setminus\Z$. Let $a_i,g\in BC^{\lambda,\lambda/4}(\R\times(0,T))$ for $i\in\{0,1,2,3,4\}$, and assume the uniform parabolicity
    \[
    \sup\limits_{x\in\R,t\in(0,T)}a_4(x,t)<0.
    \]
    Then for any $w_0\in BC^{4+\lambda}(\R)$ there is a unique solution $w\in BC^{4+\lambda,1+\lambda/4}(\R\times(0,T))$ to the problem \eqref{LCP}. Furthermore, $w$ satisfies
    \[
    \|w\|_{BC^{4+\lambda,1+\lambda/4}(\R\times(0,T))}\le C_{\lambda,\nu,T}\left(\|w_0\|_{BC^{4+\lambda}(\R)}+\|g\|_{BC^{\lambda,\lambda/4}(\R\times(0,T))}\right),
    \]
    where $\nu>0$ is such that
    \[
    \|a_i\|_{BC^{\lambda,\lambda/4}(\R\times(0,T))}\le\nu\ \textrm{for}\ i\in\{0,1,2,3,4\},\ \sup\limits_{x\in\R,t\in(0,T)}a_4(x,t)\le-\nu^{-1},
    \]
    and the constant $C$ is nondecreasing in $T$.
  \end{Proposition}

  \begin{Proposition}[Time-local parabolic Schauder estimate,\cite{Sol}*{Theorem 4.11 in Chapter \IV}]\label{locSchauder}
    Assume the same conditions as in Proposition~\ref{Schauder}. Then
    \[
    \|w\|_{BC^{4+\lambda,1+\lambda/4}(\R\times(T',T))}\le C_{\lambda,\nu,T',T}\left(\|w\|_{L^\infty(\R\times(0,T))}+\|g\|_{BC^{\lambda,\lambda/4}(\R\times(0,T))}\right)
    \]
    for all $T'\in(0,T)$.
  \end{Proposition}
  
  Furthermore, {by applying these estimates to the function
  \[
  \overline{w}(y,s):=w(T^{1/4}y,Ts),\quad (y,s)\in\R\times(0,1),
  \]
  which is a solution to the equation
  \[
  \overline{w}_t=\sum_{i=0}^4 T^{1-i/4}a_i(T^{1/4}y,Ts)\partial_y^i\overline{w}+Tg(T^{1/4}y,Ts)\ \textrm{in}\ \R\times(0,1),
  \]
  we obtain the following Schauder estimates for scaled H\"{o}lder norms.
  }
  \begin{Proposition}\label{Schauder'}
    Assume the same conditions as in Proposition~\ref{Schauder}, and additionally that $w_0\equiv 0$. Then $w$ satisfies
    \[
    \|w\|_{BC^{4+\lambda,1+\lambda/4}(\R\times(0,T))}'\le C_{\lambda,\nu}T\|g\|_{BC^{\lambda,\lambda/4}(\R\times(0,T))}',
    \]
    where $\nu>0$ is such that
    \[
    T^{1-i/4}\|a_i\|_{BC^{\lambda,\lambda/4}(\R\times(0,T))}'\le\nu\ \textrm{for}\ i\in\{0,1,2,3,4\},\ \sup\limits_{x\in\R,t\in(0,T)}a_4(x,t)\le-\nu^{-1}.
    \]
  \end{Proposition}

\begin{Proposition}\label{locSchauder'}
    Assume the same conditions as in Proposition~\ref{Schauder}. Then
    \[
    \|w\|_{BC^{4+\lambda,1+\lambda/4}(\R\times(\sigma T,T))}'\le C_{\lambda,\nu,\sigma}\left(\|w\|_{L^\infty(\R\times(0,T))}+T\|g\|_{BC^{\lambda,\lambda/4}(\R\times(0,T))}'\right)
    \]
    for all $\sigma\in(0,1)$, {where $\nu$ is as in Proposition~\ref{Schauder'}}.
  \end{Proposition}
  
  \section{Solvability of a general surface diffusion flow}

As usual, our strategy of constructing a global-in-time solution to a general surface diffusion flow \eqref{SDE'} is extending a local-in-time solution. However, since there are no available results even for the local-in-time solvability of the problem \eqref{SDE'}, we have to start with establishing it. In this section, we prove the unique solvability of the problem \eqref{SDE'} in a certain H\"{o}lder space. We define spaces $Y$ and $X_{T,u_0}$ by
  \[
  Y:=\left\{u\in C^{4+\mu}(\R);\ u_x\in L^\infty(\R)\right\},
  \]
  \[
  X_{T,u_0}:=\left\{u_0+v;\ v\in BC^{4+\mu,1+\mu/4}(\R\times(0,T)),v(\cdot,0)=0\right\},
  \]
  for $T\in(0,\infty]$, $u_0\in Y$. We also define a subset $X_{T,u_0}^R$ of $X_{T,u_0}$ by
  \[
  X_{T,u_0}^R:=\{u\in X_{T,u_0};\ \|u-u_0\|_{BC^{4+\mu,1+\mu/4}(\R\times(0,T))}\le R\}
  \]
  for $R>0$.
  \begin{Remark}
    $X_{T,u_0}$ is a complete metric space with the metric $\|u-v\|_{BC^{4+\mu,1+\mu/4}(\R\times(0,T))}$. Also, $\|(\, \cdot\, )_x\|_{BC^{3+\mu}(\R)}$ is a complete seminorm on $Y$, that is, for every sequence $\{w_n\}_{n=1}^\infty\subset Y$ satisfying $\lim\limits_{m,n\to\infty}\|(w_m-w_n)_x\|_{BC^{3+\mu}(\R)}=0$, there is a (non-unique) function $w\in Y$ such that $\lim\limits_{n\to\infty}\|(w_n-w)_x\|_{BC^{3+\mu}(\R)}=0$.
  \end{Remark}
  Using these notations, our solvability theorem for the problem \eqref{SDE'} is stated as follows.
  \begin{Theorem}\label{locexistence}
    Assume \eqref{fassump} and \eqref{u0assmup}. Then problem \eqref{SDE'} possesses a unique local-in-time solution $u\in X_{T(u_0),u_0}$. Furthermore, the maximal existence time $T(u_0)$ is bounded from below on each subset of $Y$ bounded with respect to the seminorm $\|(\, \cdot\, )_x\|_{BC^{3+\mu}(\R)}$.
  \end{Theorem}
  Our strategy of constructing a solution to the problem \eqref{SDE'} is to construct it as a perturbation from a solution to a linear parabolic equation and to use the contractivity property of lower order parabolic H\"{o}lder norms (Proposition~\ref{contraction}) to control perturbations. We first calculate
  \begin{align*}
    &\left(\frac{1}{(1+u_x^2)^{1/2}}\left(f\left(-\frac{u_{xx}}{(1+u_x^2)^{3/2}}\right)\right)_x\right)_x\\
    &=-\left(\left(\frac{u_{xxx}}{(1+u_x^2)^2}-\frac{3u_xu_{xx}^2}{(1+u_x^2)^3}\right)f'\left(-\frac{u_{xx}}{(1+u_x^2)^{3/2}}\right)\right)_x\\
    &=-A\left(u_x,u_{xx}\right)u_{xxxx}+B\left(u_x,u_{xx},u_{xxx}\right),
  \end{align*}
  where
  \begin{align*}
    A(q,r)&:=\frac{1}{(1+q^2)^2}f'\left(-\frac{r}{(1+q^2)^{3/2}}\right),\\
    B(q,r,s)&:=\left(\frac{10qrs+3r^3}{(1+q^2)^3}-\frac{18q^2r^3}{(1+q^2)^4}\right)f'\left(-\frac{r}{(1+q^2)^{3/2}}\right)\\
    &\quad +\frac{1}{(1+q^2)^{1/2}}\left(\frac{s}{(1+q^2)^{3/2}}-\frac{3qr^2}{(1+q^2)^{5/2}}\right)^2f''\left(-\frac{r}{(1+q^2)^{3/2}}\right).
  \end{align*}

    We fix $u_0\in Y$ and $M>0$ such that $\|(u_0)_x\|_{BC^{3+\mu}(\R)}\le M$. Since $A((u_0)_x,(u_0)_{xx})\in C^{2+\mu}(\R)$ and
  \[
  A((u_0)_x,(u_0)_{xx})\ge\frac{1}{(1+\|(u_0)_x\|_{L^\infty})^2}\inf\left\{f'(-\kappa):\ |\kappa|\le\left\|\frac{(u_0)_{xx}}{(1+(u_0)_x^2)^{3/2}}\right\|_{L^\infty}\right\}\ge C_M,
  \]
  it follows from Proposition~\ref{Schauder} that the Cauchy problem
  \begin{equation}\label{CP}
    v_t=-A((u_0)_x,(u_0)_{xx})v_{xxxx}+F\quad\textrm{in}\quad\R\times(0,T),\ v(\cdot,0)=0\quad\textrm{in}\quad\R\textrm{,}
  \end{equation}
  possesses a unique solution $v=:\Gamma[{F}]\in BC^{4+\mu,1+\mu/4}(\R\times(0,T))$ for all $F\in BC^{\mu,\mu/4}(\R\times(0,T))$. Furthermore, we have
  \begin{equation}\label{Gwest}
    \|\Gamma[{F}]\|_{BC^{4+\mu,1+\mu/4}(\R\times(0,T'))}\le C_{M,T}\|F\|_{BC^{\mu,\mu/4}(\R\times(0,T'))}
  \end{equation}
  for $T'\in(0,T]$.

  We also define a nonlinear map $\Lambda:\ X_{T,u_0}\to X_{T,u_0}$ by
  \begin{align*}
  \Lambda[u]:=u_0+\Gamma&\left[-A((u_0)_x,(u_0)_{xx})(u_0)_{xxxx}\right.\\
  &\left.+(A((u_0)_x,(u_0)_{xx})-A(u_x,u_{xx}))u_{xxxx}+B(u_x,u_{xx},u_{xxx})\right]
  \end{align*}
  for $u_0\in Y$.
  {The function} $\Lambda[u]$ is a sum of a fixed function (constructed as a solution to a linear parabolic equation) and nonlinear perturbation terms. It is \erase{also} clear that $u\in X_{T,u_0}$ is a solution to problem~\eqref{SDE'} if and only if $u=\Lambda[u]$.

  The following lemma controls the perturbation part of the map $\Lambda$.
  \begin{Lemma}\label{Lcontr}
    Let $u_0\in Y$ and assume that $\|(u_0)_x\|_{BC^{3+\mu}(\R)}\le M$. For any $R>0$, there is $\tau=\tau_{M,R}>0$ such that {if $T<\min(\tau_{M,R},1)$, then}
    \[
    \|\Lambda[u]-\Lambda[v]\|_{BC^{4+\mu,1+\mu/4}(\R\times(0,T))}\le\frac{1}{2}\|u-v\|_{BC^{4+\mu,1+\mu/4}(\R\times(0,T))}
    \]
    for all $u,v\in X_{{T,u_0}}^R$.
  \end{Lemma}
  \begin{proof}
We omit "$(\R\times(0,T))$" in notations of norms. We define $\varphi:=\Lambda[u]-\Lambda[v]$. It follows from the definition of $\Lambda$ that $\varphi\in BC^{4+\mu,1+\mu/4}(\R\times(0,T))$ is a solution to the problem
    \begin{equation}\label{phiprob}
      \varphi_t=-A((u_0)_x,(u_0)_{xx})\varphi_{xxxx}+H\quad\textrm{in}\quad \R\times(0,T),\ \varphi(\cdot,0)=0\quad\textrm{on}\quad\R,
    \end{equation}
    where
    \begin{align*}
      H&:=(A((u_0)_x,(u_0)_{xx})-A(u_x,u_{xx}))(u-v)_{xxxx}-(A(u_x,u_{xx})-A(v_x,v_{xx}))v_{xxxx}\\
      &\quad +B(u_x,u_{xx},u_{xxx})-B(v_x,v_{xx},v_{xxx}).
    \end{align*}
  Furthermore, since $A\in W^{3,\infty}_{\loc}(\R^2)$ and $B\in W^{2,\infty}_{\loc}(\R^3)$, it follows from Lemma~\ref{CompProd} and \ref{contraction} that
  \begin{align*}
  &\|(A((u_0)_x,(u_0)_{xx})-A(u_x,u_{xx}))(u-v)_{xxxx}\|_{BC^{\mu,\mu/4}}\\
  &\le C\|A((u_0)_x,(u_0)_{xx})-A(u_x,u_{xx})\|_{BC^{\mu,\mu/4}}\|(u-v)_{xxxx}\|_{BC^{\mu,\mu/4}}\\
  &\le C\|u-u_0\|_{BC^{2+\mu,(2+\mu)/4}}\|(u-v)_{xxxx}\|_{BC^{\mu,\mu/4}}\\
  &\le CRT^\beta\|(u-v)_{xxxx}\|_{BC^{\mu,\mu/4}},\\
  &\|(A(u_x,u_{xx})-A(v_x,v_{xx}))v_{xxxx}\|_{BC^{\mu,\mu/4}}\\
  &\le C\|A(u_x,u_{xx})-A(v_x,v_{xx})\|_{BC^{\mu,\mu/4}}\|v_{xxxx}\|_{BC^{\mu,\mu/4}}\\
  &\le C\|u-v\|_{BC^{2+\mu,(2+\mu)/4}}(\|(u_0)_x\|_{BC^{3+\mu}}+\|v-u_0\|_{BC^{4+\mu,1+\mu/4}})\\
  &\le C(M+R)T^\beta\|u-v\|_{BC^{4+\mu,1+\mu/4}},\\
  &\|B(u_x,u_{xx},u_{xxx})-B(v_x,v_{xx},v_{xxx})\|_{BC^{\mu,\mu/4}}\\
  &\le C\|u-v\|_{BC^{3+\mu,(3+\mu)/4}}\le CT^\beta\|u-v\|_{BC^{4+\mu,1+\mu/4}},
  \end{align*}
  provided that $T<1$, where $\beta>0$ and $C=C_{\mu,M,R}$ are constants. Hence
  \[
  \|H\|_{BC^{\mu,\mu/4}}\le CT^\beta\|u-v\|_{BC^{4+\mu,1+\mu/4}}.
  \]
  This together with \eqref{Gwest} implies that
  \begin{equation}\label{Lcontr'}
    \|\varphi\|_{BC^{4+\mu,1+\mu/4}}\le CT^\beta\|u-v\|_{BC^{4+\mu,1+\mu/4}}
  \end{equation}
  with $C=C(\mu,R)$. Thus, sufficiently small choices of $T$ yields the desired estimate.
  \end{proof}
  We are now in position to complete the proof of Theorem~\ref{locexistence}.
  \begin{proof}[Proof of Theorem~\ref{locexistence}]
    We first prove the existence. Let $u_0\in Y$ be such that $\|(u_0)_x\|_{BC^{3+\mu}}\le M$. Since
    \[
    \Lambda[u_0]-u_0=\Gamma[-A((u_0)_x,(u_0)_{xx})(u_0)_{xxxx}+B(((u_0)_x,(u_0)_{xx},(u_0)_{xxx}))],
    \]
    it follows from \eqref{CompProd} and \eqref{Gwest} that there is a constant $C_*$ depending only on $M$ such that
    \[
    \|\Lambda[u_0]-u_0\|_{BC^{4+\mu,1+\mu/4}(\R\times(0,1))}\le C_*.
    \]
    Taking $R>2C_*$ and $\tau>0$ as in Lemma~\ref{Lcontr}, we see that
    \begin{align*}
    \|\Lambda[v]-u_0\|_{BC^{4+\mu,1+\mu/4}}&\le\|\Lambda[v]-\Lambda[u_0]\|_{BC^{4+\mu,1+\mu/4}}+\|\Lambda[u_0]-u_0\|_{BC^{4+\mu,1+\mu/4}}\\
    &\le \frac{1}{2}\|v-u_0\|_{BC^{4+\mu,1+\mu/4}}+C_*<\frac{R}{2}+\frac{R}{2}=R
    \end{align*}
    for all $v\in X_{T,u_0}^R$, provided that $T<\min\{\tau,1\}$. Hence $\Lambda(X_{T,u_0}^R)\subset X_{T,u_0}^R$. This together with Lemma~\ref{Lcontr} implies that there is a unique fixed point of $\Lambda$ in $X_{T,u_0}^R$, which is a solution to problem \eqref{SDE'}, provided that $T<\min\{\tau,1\}$. Furthermore, since $\tau$ depends only on $\mu$ and $M$, the maximal existence time $T(u_0)$ is bounded from below on each subset of $Y$ bounded with respect to the seminorm $\|(\,\cdot\,)_x\|_{BC^{3+\mu}}$.

  We finally prove the uniqueness. Let $u\in X_{T,u_0}$, $\tilde{u}\in X_{\tilde{T},u_0}$, $T\le\tilde{T}$, be solutions to \eqref{SDE'}, and $R>0$ be such that $u,\tilde{u}\in X_{T,u_0}^R$. Assume on the contrary that
  \[
  T':=\sup\left\{t\in(0,T];\ u=\tilde{u}\ \textrm{on}\ \R\times(0,t)\right\}<T.
  \]
  By replacing $u_0$, $u\in X_{T,u_0}^R$ by $u(\cdot,T')$, $u(\cdot,\cdot+T')\in X_{T-T',u(\cdot,T')}^R$ respectively, we may assume that $T'=0$.
  
  Let $\tau$ be as in Lemma~\ref{Lcontr}. Then
  \[
  \|u-\tilde{u}\|_{BC^{4+\mu,1+\mu/4}(\R\times(0,\tau))}=\|\Lambda[u]-\Lambda[\tilde{u}]\|_{BC^{4+\mu,1+\mu/4}(\R\times(0,\tau))}\le \frac{1}{2}\|u-\tilde{u}\|_{BC^{4+\mu,1+\mu/4}(\R\times(0,\tau))}.
  \]
  This implies that $u=\tilde{u}$ in $\R\times(0,\tau)$, which contradicts $T'=0$.
  \end{proof}
  \section{Decay estimates for a perturbed biharmonic heat equation}
  In this section, we consider a linear perturbed biharmonic heat equation of the form
  \begin{equation}\label{PBHE}
    v_t=-((1-\alpha)v_{xx})_{xx}+F_{xx}\ \textrm{in}\ \R\times(0,T),\ v(\cdot,0)=v_0\ \textrm{in}\ \R,
  \end{equation}
  and derive decay estimates for the solution $v$. These estimates will play an essential role to deal with the higher-order perturbation of the equation \eqref{SDE'}.
  
  It is easily seen that problem \eqref{PBHE} is equivalent to the integral equation
  \begin{equation}\label{IE}
  v(t)=e^{-t\partial_x^4}v_0+\int_0^t \partial_x^2e^{-(t-s)\partial_x^4}[\alpha v_{xx}(s)]\,ds+\int_0^t \partial_x^2e^{-(t-s)\partial_x^4}F(s)\,ds,
  \end{equation}
  provided that $v\in C^{4+\mu,1+\mu/4}_{\textrm{time loc}}(\R\times(0,T))\cap L^\infty(\R\times(0,T))$. We assume on the coefficient function $\alpha$ that
  \begin{itemize}
    \item[(A)] $\|\alpha\|_{BC^{2+\mu,(2+\mu)/4}(\R\times(t/2,t))}'\le \delta$ for all $t\in(0,T)$,
  \end{itemize}where $\delta>0$ is a small constant and $\mu\in(0,1)$.

  We also define the weighted H\"{o}lder norms $\|\cdot\|_{Z^k_T}$ and the function space $Z^k_T$ as follows.
  \begin{Definition}\label{defZkt}
    For $k\in\Z_{\ge 0}$ and $T\in(0,\infty]$, the space $Z^k_T$ consists of $v\in C^{k+\mu,(k+\mu)/4}_{\textrm{time loc}}(\R\times(0,T))$ such that the value
  \[
\|v\|_{Z^k_T}:=\sup_{t\in(0,T)}\left(\|v\|_{BC^{k+\mu,(k+\mu)/4}(\R\times(t/2,t))}'+(1+t)^{1/4}\|v_x\|_{BC^{k-1+\mu,(k-1+\mu)/4}(\R\times(t/2,t))}'\right)
  \]
  is finite.  
\end{Definition}
\begin{Remark}\label{rmkZkt}
We will use this norm on the estimate of the spatial derivative $v=u_x$ of the solution $u$ to the problem~\eqref{SDE'} in Section~5.
 The first term is invariant under the scaling $v\mapsto v(\sigma x,\sigma^4t)$ for $\sigma>1$. The second term makes the spatial derivative $v_x$ of $v\in Z^k_T$ bounded near $t=0$, while it is equivalent to the first term for large $t$. Thanks to the second term, the curvature $\kappa=u_{xx}/(1+u_x^2)^{3/2}$ is bounded if $v=u_x\in Z^k_T$, which enables us to deal with the curvature nonlinearity $f(-\kappa)$. Furthermore, the asymptotic scale invariance for large $t$ makes this norm suitable for asymptotic analysis of self-similar type for problem~\eqref{SDE'}.
\end{Remark}
  Note that $Z^k_T$ is complete with respect of the norm $\|\cdot\|_{Z^k_T}$.

  \begin{Lemma}\label{123'}
    Suppose that $v_0\in W^{1,\infty}(\R)$, and that $F$ satisfies
  \begin{equation}\label{Fnorm'}
  \|F\|_{BC^{\mu,\mu/4}(\R\times(t/2,t))}'\le \varepsilon t^{-1/4}(1+t)^{-1/4}
  \end{equation}
  for all $t\in(0,T)$. Then the function
  \begin{equation}\label{Phi}
  \Phi(x,t):=e^{-t\partial_x^4}v_0+\int_0^t \partial_x^2 e^{-(t-s)\partial_x^4}F(s)\,ds
  \end{equation}
  on $\R\times(0,T)$ solves \eqref{IE} and satisfies
  \begin{equation}\label{Phinorm}
    \|\Phi\|_{Z^2_T}\le C(\|v_0\|_{W^{1,\infty}}+\varepsilon).
  \end{equation}
  \end{Lemma}
  \begin{proof}
  We decompose $\Phi$ into $I+\II+\III$, where
        \[
  I:=e^{-t\partial_x^4}v_0,\
  \II:=\int_0^{\tilde{t}} \partial_x^2e^{-(t-s)\partial_x^4}F(s)\,ds,\ \III:=\int_{\tilde{t}}^t \partial_x^2e^{-(t-s)\partial_x^4}F(s)\,ds.
  \]
  
      We fix $\tilde{t}\in(0,T/4)$ and consider the range $t\in(2\tilde{t},4\tilde{t})$. We first observe that
  \begin{equation}\label{I}
    \|I\|_{BC^{\lambda,\lambda/4}(\R\times(a,b))}'\le C\|v_0\|_{L^\infty},\ \|I_x\|_{BC^{\lambda,\lambda/4}(\R\times(t/2,t))}'\le C(1+t)^{-1/4}\|v_0\|_{W^{1,\infty}},
  \end{equation}
  for all $\lambda\ge 0$.  

  We next estimate $\II$.
  By
  \[
\partial_x^k\partial_t^\ell\II=\int_0^{\tilde{t}}\partial_x^{k+4\ell+2}e^{-(t-s)\partial_x^4}F(s)\,ds
  \]
  we see that
    \begin{align*}
      \|\partial_x^k\partial_t^\ell\II\|_{L^\infty}(t)&\le C\int_0^{\tilde{t}}(t-s)^{-(k+4\ell+2)/4}\|F\|_{L^\infty}(s)\,ds\\
      &\le Ct^{-(k+4\ell+2)/4}\int_0^{t/2}\varepsilon s^{-1/4}(1+s)^{-1/4}\,ds\\
      &\le C\varepsilon t^{-(k+4\ell+2)/4}\cdot t^{3/4}(1+t)^{-1/4}= C\varepsilon t^{(1-k-4\ell)/4}(1+t)^{-1/4}
    \end{align*}
  for all $k,\ell\in\Z_{\ge 0}$. By interpolation, we also obtain
    \[
    [\partial_x^k\partial_t^\ell\II]_{C^{\mu,\mu/4}(\R\times(t/2,t))}\le C\varepsilon t^{(1-k-4\ell-\mu)/4}(1+t)^{-1/4}
    \]
  for all $\mu\in(0,1)$. Thus it holds that
  \begin{equation}\label{II-2}
  \begin{gathered}
      \|\II\|_{BC^{\lambda,\lambda/4}(\R\times(t/2,t))}'\le C\varepsilon t^{1/4}(1+t)^{-1/4},\\
      \|\II_x\|_{BC^{\lambda-1,(\lambda-1)/4}(\R\times(t/2,t))}'\le C\varepsilon (1+t)^{-1/4},
  \end{gathered}
  \end{equation}
  for all $\lambda\ge 0$.

  We finally estimate $C^{2+\mu,(2+\mu)/4}$ norms of derivatives of $\III$. We set
  \[
  \widetilde{\III}:=\int_{\tilde{t}}^t e^{-(t-s)\partial_x^4}F(s)\,ds,
  \]
  so that $\partial_x^2\widetilde{\III}=\III$. 
  Then Proposition~\ref{Schauder'} implies that
  \[
  \|\widetilde{\III}\|_{BC^{4+\mu,1+\mu/4}(\R\times(\tilde{t},t))}'\le Ct\|F\|_{BC^{\mu,\mu/4}(\R\times(\tilde{t},t))}'\le C\varepsilon t^{3/4}(1+t)^{-1/4}.
  \]
  Combining this with $\III=\partial_x^2\widetilde{\III}$, we deduce that
  \begin{equation}\label{III-2'}
  \begin{gathered}
  \|\III\|_{BC^{2+\mu,(2+\mu)/4}(\R\times(\tilde{t},t))}'\le C\varepsilon t^{1/4}(1+t)^{-1/4}\le C\varepsilon,\\
  \|\III_x\|_{BC^{1+\mu,(1+\mu)/4}(\R\times(\tilde{t},t))}'\le 
  t^{-1/4}\|\III\|_{BC^{2+\mu,(2+\mu)/4}(\R\times(\tilde{t},t))}'\le C\varepsilon(1+t)^{-1/4}.
  \end{gathered}
  \end{equation}
  Adding \eqref{I}, \eqref{II-2}, and \eqref{III-2'}, we obtain
  \begin{gather*}
  \|\Phi\|_{BC^{2+\mu,(2+\mu)/4}(\R\times(t/2,t))}'\le C(\|u_0\|_{W^{1,\infty}}+\varepsilon),\\
  \|\Phi_x\|_{BC^{1+\mu,(1+\mu)/4}(\R\times(t/2,t))}'\le C(\|u_0\|_{W^{1,\infty}}+\varepsilon)(1+t)^{-1/4},
  \end{gather*}
  and complete the proof of Lemma~\ref{123'}.
  \end{proof}
               
  \begin{Lemma}\label{123}
    Suppose that $v_0\in W^{1,\infty}(\R)$, and that $F$ satisfies
      \begin{equation}\label{Fnorm}
  \|F\|_{BC^{2+\mu,(2+\mu)/4}(\R\times(t/2,t))}'\le \varepsilon t^{-1/4}(1+t)^{-1/4}\erase{,}
  \end{equation}
  for all $t\in(0,T)$. Then the function $\Phi$ defined in \eqref{Phi} satisfies
  \begin{equation}\label{Phinorm'}
    \|\Phi\|_{Z^4_T}\le C(\|v_0\|_{W^{1,\infty}}+\varepsilon).
  \end{equation}
  \end{Lemma}
  \begin{proof}
  Proposition~\ref{Schauder'} {yields} the estimate
    \begin{equation}\label{III-2}
    \|\III\|_{BC^{4+\mu,1+\mu/4}(\R\times(\tilde{t},t))}'\le Ct\|F_{xx}\|_{BC^{\mu,\mu/4}(\R\times(\tilde{t},t))}'\le Ct^{1/4}(1+t)^{-1/4}.
    \end{equation}
    Adding \eqref{I}, \eqref{II-2} with $\lambda=4+\mu$ and \eqref{III-2}, we obtain \eqref{Phinorm'}. This completes the proof of Lemma~\ref{123}.
    \end{proof}
  Now we are ready to state our linear decay estimates for the equation \eqref{PBHE}.
\begin{Theorem}\label{linearestimate}
  Suppose that $v_0\in W^{1,\infty}(\R)$. Assume (A) and \eqref{Fnorm}. Then there are constants $\delta_0>0$ and $C_0>0$ independent of $T$, such that if $\delta<\delta_0$, then there is a unique solution to $v\in Z^4_T$ be a solution to problem \eqref{PBHE},
  and $v$ satisfies
  \begin{equation}\label{vest}
  \|v\|_{Z^4_T}\le C(\|v_0\|_{W^{1,\infty}}+\varepsilon).
  \end{equation}
  Furthermore, if $v\in Z^2_T$ is a solution to problem \eqref{IE}, then $v$ belongs to $Z^4_T$ and solves problem \eqref{PBHE}.
\end{Theorem}
\begin{proof}
  We define a linear operator $L$ on $Z^2_T$ by
  \[
  L[v](t):=\int_0^t e^{-(t-s)\partial_x^4}(\alpha v_{xx})(s)\,ds.
  \]

  The inequality \eqref{Phinorm'} implies that $\Phi\in Z^4_T$ and $\|\Phi\|_{Z^4_T}\le C(\|v_0\|_{W^{1,\infty}}+\varepsilon)$. Furthermore, we have
  \begin{equation}\label{avxx}
  \begin{aligned}
    &\|\alpha v_{xx}\|_{BC^{k-1+\mu,(k-1+\mu)/4}(\R\times(t/2,t))}'\\
    &\le C\|\alpha\|_{BC^{k-1+\mu,(k-1+\mu)/4}(\R\times(t/2,t))}'\cdot t^{-1/4}\|v_x\|_{BC^{k-1+\mu,(k-1+\mu)/4}(\R\times(t/2,t))}'\\
    &\le C\delta\|v\|_{Z^k_T}t^{-1/4}(1+t)^{-1/4}
  \end{aligned}
  \end{equation}
 for $k\in\{2,4\}$. This together with Lemmas~\ref{123'}~and~\ref{123} with $v_0=0$ and $F=\alpha v_{xx}$ implies that
  \[
  \|L[v]\|_{Z^2_T}\le C\delta\|v\|_{Z^2_T},\ \|L[v]\|_{Z^4_T}\le C\delta\|v\|_{Z^4_T}.
  \]
  
  In particular, if $\delta>0$ is sufficiently small, $\operatorname{id}-L$ is an automorphism on $Z^k_T$ for $k\in\{2,4\}$. Since \eqref{PBHE} is equivalent to $v-L v=\Phi$ and $\|\Phi\|_{Z^4_T}\le C(\|v_0\|_{W^{1,\infty}}+\varepsilon)$, we deduce the unique existence of a solution to problem \eqref{PBHE} in $Z^4_T$ and the estimate \eqref{vest}. Furthermore, if $\overline{v}\in Z^2_T$ is a solution to problem \eqref{IE}, $\overline{v}$ coincides with the solution $v\in Z^4_T$ to problem \eqref{PBHE} since $\operatorname{id}-L$ is an automorphism on $Z^2_T$. The proof of Theorem~\ref{linearestimate} is complete.
\end{proof}

We finally mention the following higher regularity of solutions to problem \eqref{PBHE}.

\begin{Lemma}\label{higherreg}
  Let $k\in\Z_{\ge 3}$. Let $\delta_0$ be as in Theorem~\ref{linearestimate} and assume (A) with $\delta<\delta_0$. Assume also that
  \begin{equation}\label{A2} 
    \sup_{t\in(0,T)}\|\alpha\|_{BC^{k+\mu,(k+\mu)/4}(\R\times(t/2,t))}'<\infty.
  \end{equation}
  Suppose that $v_0\in W^{1,\infty}(\R)$, and that $F$ satisfies
  \begin{equation}\label{highFnorm}
    \|F\|_{BC^{k+\mu,(k+\mu)/4}(\R\times(t/2,t))}'\le \varepsilon t^{-1/4}(1+t)^{-1/4}
  \end{equation}
  for all $t\in(0,T)$. Let $v\in Z^4_T$ be a solution to equation $v_t=-((1-\alpha)v_{xx})_{xx}+F_{xx}$. Then
  \begin{equation}\label{highervest}
    \|v\|_{Z^{k+2}_T}\le C(\|v_0\|_{L^\infty(\R)}+\varepsilon).
  \end{equation}
\end{Lemma}
\begin{proof}
  We observe that the equation $v_t=-((1-\alpha)v_{xx})_{xx}$ is equivalent to
  \[
  v_t=-(1-\alpha)v_{xxxx}+2\alpha_xv_{xxx}+\alpha_{xx}v_{xx}+F_{xx}.
  \]
  It follows from \eqref{highFnorm} that
    \[
    \|F_{xx}\|_{BC^{k-1+\mu,(k-1+\mu)/4}(\R\times(t/4,t))}'\le C\varepsilon t^{-3/4}(1+t)^{-1/4}.
  \]
  Furthermore, the assumption (A2) implies that
  \begin{align*}
    \|1-\alpha\|_{BC^{k-1+\mu,(k-1+\mu)/4}(\R\times(t/4,t))}'&\le C,\\
    \|\alpha_x\|_{BC^{k-1+\mu,(k-1+\mu)/4}(\R\times(t/4,t))}'&\le C,\\
    \|\alpha_{xx}\|_{BC^{k-1+\mu,(k-1+\mu)/4}(\R\times(t/4,t))}'&\le C,
  \end{align*}
  By Proposition~\ref{locSchauder'}, we see that
  \begin{align*}
  \|v\|_{BC^{k+2+\mu,(k+2+\mu)/4}(\R\times(t/2,t))}'&\le C\left(t\cdot\varepsilon t^{-3/4}(1+t)^{-1/4}+\|v\|_{L^\infty(\R\times(t/4,t))}\right)\le C(\|v_0\|_{W^{1,\infty}}+\varepsilon).
  \end{align*}
  For the estimate of $\|v_x\|_{BC^{k+1+\mu,(k+1+\mu)/4}}'$, we consider an equation which the function $w=v_x$ solves, namely
  \[
  w_t=-(1-\alpha)w_{xxxx}+3\alpha_xw_{xxx}+3\alpha_{xx}w_{xx}+\alpha_{xxx}w_x+F_{xxx}.
  \]
  Since
  \[
   \|F_{xxx}\|_{BC^{k-3+\mu,(k-3+\mu)/4}(\R\times(t/4,t))}'\le C\varepsilon t^{-1}(1+t)^{-1/4},
  \]
   Proposition~\ref{locSchauder'} and \eqref{highFnorm} yield
  \begin{align*}
     \|v_x\|_{BC^{k+1+\mu,(k+1+\mu)/4}(\R\times(t/4,t))}'&\le C\left(t\cdot\varepsilon t^{-1}(1+t)^{-1/4}+\|v_x\|_{L^\infty(\R\times(t/2,t))}\right)\\
     &\le C(\|v_0\|_{W^{1,\infty}}+\varepsilon)(1+t)^{-1/4}.
  \end{align*}
  Thus \eqref{highervest} follows.
\end{proof}

\section{Nonlinear decay estimates for a general surface diffusion flow}
  We consider decay estimates for a solution $u\in X_{T,u_0}$ to problem \eqref{SDE'} via linear estimates in Section~4 and elementary estimates for perturbation terms. These enable us to extend $u$ to a global-in-time solution. We assume that \eqref{fassump} and \eqref{u0assmup} throughout the section.
  
  Instead of directly estimating $u$, we estimate the spatial derivative $v=u_x$ of $u$. We define functions $v_0$, $\alpha$, and $F$ by $v_0=(u_0)_x$ and
  \begin{equation}\label{alpha}
  \alpha:=1-\frac{1}{(1+v^2)^2}f'\left(-\frac{v_x}{(1+v^2)^{3/2}}\right),
  \end{equation}
  \begin{equation}\label{F}
  F=\frac{3vv_x^2}{(1+v^2)^3}f'\left(-\frac{v_x}{(1+v^2)^{3/2}}\right),
  \end{equation}
  so that $v$ solves the integral equation
  \begin{equation}\label{SDE''}
  v(t)=e^{-t\partial_x^4}v_0+\int_0^t\partial_x^2 e^{-(t-s)\partial_x^4}[\alpha v_{xx}(s)]\,ds+\int_0^t \partial_x^2e^{-(t-s)\partial_x^4}F(s)\,ds.
  \end{equation}
  We first estimate each nonlinear component of the equation \eqref{SDE''} by $\|v\|_{Z^k_T}$ and verify that we are in a situation that Theorem~\ref{linearestimate} applies.
  \begin{Remark}\label{rmkZkt2}
      Thanks to the adoption of the scaled H\"{o}lder norms $\|\cdot\|_{BC^{k+\mu,(k+\mu)/4}(\R\times(t/2,t))}'$ and weighted H\"{o}lder norms $\|\cdot\|_{Z^k_T}$, we can obtain adequate estimates of derivatives of $v$ without expanding them, which makes our calculations much simpler than using the usual H\"{o}lder norm $\|\cdot\|_{BC^{k+\mu,(k+\mu)/4}(\R\times(t/2,t))}$ (See Remark~\ref{Holrmk} and Proposition~\ref{CompProd})
  \end{Remark}
  \begin{Lemma}\label{kappa}
    Suppose that $\mu\in(0,1)$, $R>0$, and $\|v\|_{Z^k_T}\le R$.
  Define
  \[
  \kappa:=\frac{v_x}{(1+v^2)^{3/2}}.
  \]
  Then
  \begin{equation}\label{kappaest}
  \|\kappa\|_{BC^{k-1+\mu,(k-1+\mu)/4}(\R\times(t/2,t))}'\le C\|v\|_{Z^k_T} (1+t)^{-1/4},
  \end{equation}
  \begin{equation}\label{fkappaest}
    \|f'(-\kappa)-1\|_{BC^{k-1+\mu,(k-1+\mu)/4}(\R\times(t/2,t))}'\le C\|v\|_{Z^k_T}(1+t)^{-1/4},
  \end{equation}
  for all $t\in(0,T)$.
  \end{Lemma}
  \begin{proof}
    We fix $t\in(0,T)$ and omit ``$(\R\times(t/2,t))$'' in notations of norms. It follows from Proposition~\ref{CompProd} that
  \begin{equation}\label{gammaest}
    \begin{aligned}
    \|(1+v^2)^\gamma\|_{BC^{k-1+\mu,(k-1+\mu)/4}}'&\le C,\\
    \|(1+v^2)^\gamma-1\|_{BC^{k-1+\mu,(k-1+\mu)/4}}'&\le C\|v^2\|_{BC^{k-1+\mu,(k-1+\mu)}/4}'\\
    &\le C\|v\|_{BC^{k-1+\mu,(k-1+\mu)/4}}^{\prime\ 2}\le C\|v\|_{Z^{k-1}_T}^2,
    \end{aligned}
  \end{equation}
  for all $\gamma\in\R$. Using Proposition~\ref{CompProd} again yields
  \begin{align*}
  \|\kappa\|_{BC^{k-1+\mu,(k-1+\mu)/4}}'&\le C\|v_x\|_{BC^{k-1+\mu,(k-1+\mu)/4}}'\|(1+v^2)^{-3/2}\|_{BC^{k-1+\mu,(k-1+\mu)/4}}'\\
  &\le C\|v\|_{Z^k_T}(1+t)^{-1/4},\\
  \|f'(-\kappa)-1\|_{BC^{k-1+\mu,(k-1+\mu)/4}}'&\le C\|\kappa\|_{BC^{k-1+\mu,(k-1+\mu)/4}}'\le C\|v\|_{Z^k_T}(1+t)^{-1/4}.
  \end{align*}
 The proof of Lemma~\ref{kappa} is complete.
  \end{proof}
  \begin{Lemma}\label{lemalphaest}
  Assume the same conditions as in Lemma~\ref{kappa}.  Define a function $\alpha$ as in \eqref{alpha}. Then
  \begin{equation}\label{alphaest}
  \|\alpha\|_{BC^{k-1+\mu,(k-1+\mu)/4}(\R\times(t/2,t))}'\le C\|v\|_{Z^k_T}
  \end{equation}
  for all $t\in(0,T)$.
\end{Lemma}
\begin{proof}
   We fix $t\in(0,T)$ and omit ``$(\R\times(t/2,t))$'' in notations of norms. 
   We deduce from Proposition~\ref{CompProd}, \eqref{fkappaest}, and \eqref{gammaest} that
    \begin{align*}
      \|\alpha\|_{BC^{k-1+\mu,(k-1+\mu)/4}}'
      &\le C\left(\|f'(-\kappa)\|_{BC^{k-1+\mu,(k-1+\mu)/4}}'+1\right)\|(1+v^2)^{-2}-1\|_{BC^{k-1+\mu,(k-1+\mu)/4}}'\\
      &\quad +C\left(\|(1+v^2)^{-2}\|_{BC^{k-1+\mu,(k-1+\mu)/4}}'+1\right)\|f'(-\kappa)-1\|_{BC^{k-1+\mu,(k-1+\mu)/4}}'\\
      &\le C\|v\|_{Z^k_T}^2+C\|v\|_{Z^k_T}(1+t)^{-1/4}\le C\|v\|_{Z^k_T}.
    \end{align*}
The proof of Lemma~\ref{lemalphaest} is complete.
\end{proof}
\begin{Lemma}\label{lemFest}
  Assume the same conditions as in Lemma~\ref{kappa}. Define a function $F$ as in \eqref{F}. Then
  \begin{equation}\label{Fest}
    \|F\|_{BC^{k-1+\mu,(k-1+\mu)/4}(\R\times(t/2,t))}'\le C\|v\|_{Z^k_T}^3(1+t)^{-1/2},
  \end{equation}
  for all $t\in(0,T)$.
\end{Lemma}
\begin{proof}
    We fix $t\in(0,T)$ and omit ``$(\R\times(t/2,t))$'' in notations of norms. Since $F=v\kappa^2 f'(-\kappa)$, it follows from Proposition~\ref{CompProd}, \eqref{kappaest}, and \eqref{fkappaest} that
    \begin{align*}
      \|F\|_{BC^{k-1+\mu,(k-1+\mu)/4}}'
      &\le C\|v\|_{BC^{k-1+\mu,(k-1+\mu)/4}}'\|\kappa\|_{BC^{k-1+\mu,(k-1+\mu)/4}}^{\prime\ 2}\|f'(-\kappa)\|_{BC^{k-1+\mu,(k-1+\mu)/4}}'\\
      &\le C\|v\|_{Z^k_T}^3(1+t)^{-1/2}.
    \end{align*}
The proof of Lemma~\ref{lemFest} is complete.
\end{proof}

\begin{Lemma}\label{vdecayest}
  Suppose that $v\in Z^3_T$ solves \eqref{SDE''}. Then there are constants $\delta_0>0$ and $C_1>0$ such that if $\|v\|_{Z^3_T}\le\delta_0$, then $v$ belongs to $Z^k_T$ for all $k\in\Z_\ge 4$. Furthermore, $v$ satisfies
  \begin{equation}\label{vdecayest2}
  \|v\|_{Z^{k}_T}\le C_1\|v_0\|_{W^{1,\infty}}.
  \end{equation}
\end{Lemma}
\begin{proof}
  We prove \eqref{vdecayest2} by the induction on $k$. We first prove the case $k=4$. We assume that $\|v\|_{Z^3_T}\le 1$. It follows from \eqref{alphaest}, and \eqref{Fest} that
  \begin{gather*}
  \|\alpha\|_{BC^{2+\mu,(2+\mu)/4}(\R\times(t/2,t))}'\le C\|v\|_{Z^3_T},\\
  \|F\|_{BC^{2+\mu,(2+\mu)/4}(\R\times(t/2,t))}'\le C\|v\|_{Z^3_T}^3(1+t)^{-1/2},
  \end{gather*}
  for all $t\in(0,T)$. By Theorem~\ref{linearestimate}, $v$ belongs to $Z^4_T$ and satisfies
  \[
  \|v\|_{Z^4_T}\le C\left(\|v_0\|_{W^{1,\infty}}+\|v\|_{Z^3_T}^3\right)\le C\left(\|v_0\|_{W^{1,\infty}}+\|v\|_{Z^3_T}^2\|v\|_{Z^4_T}\right).
  \]
  This implies that $\|v\|_{Z^4_T}\le C\|v_0\|_{W^{1,\infty}}$ if $\|v\|_{Z^3_T}$ is sufficiently small. The proof in the case $k=4$ is complete.
  
  We next assume that $\eqref{vdecayest2}$ holds in the case $k=k_0\in\Z_{\ge 4}$. Then \eqref{alphaest} and \eqref{Fest} imply that
  \begin{gather*}
      \|\alpha\|_{BC^{k_0-1+\mu,(k_0-1+\mu)/4}(\R\times(t/2,t))}'\le C\|v\|_{Z^{k_0}_T},\\
  \|F\|_{BC^{k_0-1+\mu,(k_0-1+\mu)/4}(\R\times(t/2,t))}'\le C\|v\|_{Z^{k_0}_T}^3(1+t)^{-1/2}.
  \end{gather*}
  Thus Lemma~\ref{higherreg} yields \eqref{vdecayest2} with $k=k_0+1$.
 By induction, we obtain $v\in Z^k_T$ for all $k\in Z^k_T$ for all $k\in\Z_{\ge 4}$, and complete the proof of Lemma~\ref{vdecayest}.
\end{proof}

\begin{Lemma}\label{utrap}
  Suppose that $u\in Y$. There is a number $\varepsilon_0>0$ such that if $\|(u_0)_x\|_{W^{1,\infty}}<\varepsilon_0$, then the solution $u\in X_{u_0,T(u_0)}$ to problem \eqref{SDE'} satisfies $u_x\in Z^k_{T(u_0)}$ for all $k\in\Z_{\ge 4}$, and the estimate \eqref{vdecayest2}.
\end{Lemma}
\begin{proof}
    Let $\delta_0$ and $C_1$ be as in Lemma~\ref{vdecayest}. We define $\varepsilon_0:=(2+2C_1)^{-1}\delta_0$. It suffices to prove that $\|u_x\|_{Z^3_{T(u_0)}}\le\delta_0$, since this together with Lemma~\ref{vdecayest} implies \eqref{vdecayest2}.
    
    It follows from the contractivity property of lower order H\"{o}lder norms that
    \begin{align*}
    \|u_x\|_{Z^3_T}&\le\|u_x\|_{BC^{3+\mu,(3+\mu)/4}(\R\times(0,T))}'+\|u_{xx}\|_{BC^{2+\mu,(2+\mu)/4}(\R\times(0,T))}'\\
    &\le \|u_x\|_{L^\infty(\R\times(0,T))}+\|u_{xx}\|_{L^\infty(\R\times(0,T))}\\
    &\quad +2T^{1/4}\sum_{\ell=0}^3\|\partial_x^{\ell+1}u\|_{L^\infty(\R\times(0,T))}+2T^{(1+\mu)/4}[u_x]_{C^{3+\mu,(3+\mu)/4}(\R\times(0,T))}\\
    &\le \|(u_0)_x\|_{W^{1,\infty}(\R)}+CT^\beta\|u_x\|_{BC^{3+\mu,(3+\mu)/4}(\R\times(0,T))}\le \varepsilon_0+CT^\beta\|u_x\|_{BC^{3+\mu,(3+\mu)/4}(\R\times(0,T))}
    \end{align*}
    for $T\in(0,1)$, where $\beta\in(0,1/4]$ is a constant. In particular, it follows that $\|u_x\|_{Z^3_{T_0}}<2\varepsilon_0<\delta_0$ for a sufficiently small $T_0>0$. Thus    
    \[
    \|u_x\|_{Z_{T_0}^4}\le C_1\varepsilon_0<\delta_0
    \]
    by the choice of $\delta_0$. This implies that
    \[
  T':=\sup\{T\in(0,T(u_0));\ \textrm{$u$ satisfies $\|u_x\|_{Z_{T}^3}\le\delta_0$}\}>0.
  \]
    It remains to prove that $T'=T(u_0)$. Assume on the contrary that $T'<T(u_0)$. Then $v=u_x$ satisfies $\|v\|_{Z_{T'}^4}\le\delta_0$ and therefore
    $\|v\|_{Z_{T'}^4}\le C_1\varepsilon_0$ by Lemma~\ref{vdecayest}.
    In particular, we have
  \[
  \|v\|_{BC^{4+\mu,1+\mu/4}(\R\times(T'/2,T'))}'+(1+T')^{1/4}\|v_x\|_{C^{3+\mu,(3+\mu)/4}(\R\times(T'/2,T'))}'\le C_1\varepsilon_0.
  \]
  Since $v\in C^{4+\mu,1+\mu/4}_{\textrm{time loc}}(\R\times(0,T(u_0)))$, the contractivity property of lower order H\"{o}lder norms implies that for some small $\tau>0$, we have
    \[
    \|v\|_{BC^{3+\mu,(3+\mu)/4}(\R\times(t/2,t))}'+(1+t)^{1/4}\|v_x\|_{C^{2+\mu,(2+\mu)/4}(\R\times(t/2,t))}'\le 2C_1\varepsilon_0
    \]
    for $t\in[T',T'+\tau)$. This together with $\|v\|_{Z^4_{T'}}\le C_1\varepsilon_0$ implies that $\|v\|_{Z^3_{T'+\tau}}\le 2C_1\varepsilon_0$. Since $2C_1\varepsilon_0<\delta_0$, this contradicts the choice of $T'$.
\end{proof}
Now we are ready to complete the proof of Theorem~\ref{GE}.
\begin{proof}[Proof of Theorem~\ref{GE}]
  Let $\varepsilon_0>0$ be as in Lemma~\ref{utrap}. We first prove that $T(u_0)=\infty$. Assume on the contrary that $T(u_0)<\infty$. By Lemma~\ref{utrap}, the set $\{u_x(t)\}_{t\in [T(u_0)/2,T(u_0))}$ is a bounded subset of $BC^{3+\mu}(\R)$. This together with Theorem~\ref{locexistence} implies that there is $\tau>0$ such that problem \eqref{SDE'} with $u_0$ replaced with $u(t_0)$ possesses a solution in $X_{\tau,u(t_0)}$ for any $t_0\in[T(u_0)/2,T(u_0))$. Taking $t_0$ so that $t_0+\tau>T(u_0)$, we see the solution $u$ is extended to $\widetilde{u}\in X_{t_0+\tau,u_0}$, which contradicts the definition of $T(u_0)$.
  
    We finally estimate $u_t$. It follows from \eqref{SDE'}, \eqref{fkappaest}, \eqref{gammaest}, and \eqref{vdecayest2} that
  \begin{equation}\label{utest}
  \begin{aligned}
    &\|u_t\|_{BC^{k-3+\mu,(k-3+\mu)/4}(\R\times(t/2,t))}'=\left\|\left(\frac{1}{(1+u_x^2)^{1/2}}\left(f(-\kappa)-1\right)_x\right)_x\right\|_{BC^{k-3+\mu,(k-3+\mu)/4}(\R\times(t/2,t))}'\\
    &\le t^{-1/4}\left\|\frac{1}{(1+u_x^2)^{1/2}}\left(f(-\kappa)-1\right)_x\right\|_{BC^{k-2+\mu,(k-2+\mu)/4}(\R\times(t/2,t))}'\\
    &\le Ct^{-1/4}\left\|(1+u_x^2)^{-1/2}\right\|_{BC^{k-2+\mu,(k-2+\mu)/4}(\R\times(t/2,t))}'\cdot t^{-1/4}\left\|f(-\kappa)-1\right\|_{BC^{k-2+\mu,(k-2+\mu)/4}(\R\times(t/2,t))}'\\
    &\le C\|v\|_{Z^k_\infty} t^{-1/2}(1+t)^{-1/4}\le C\|(u_0)_x\|_{W^{1,\infty}} t^{-1/2}(1+t)^{-1/4}.
  \end{aligned}
  \end{equation}
  This estimate together with \eqref{vdecayest2} clearly implies \eqref{Thmudecay}. The proof of Theorem~\ref{GE} is complete.
\end{proof}
\section{Proof of Theorem~\ref{ConvSSS}}
We are now in position to complete the  proof of Theorem~\ref{ConvSSS}. All we have to do is to justify a formal limit argument in Remark~\ref{expect} to prove the convergence to a self-similar solution to the equation \eqref{SSS}.
\begin{proof}[Proof of Theorem~\ref{ConvSSS}] 
  Let $u^\sigma$ be as in the statement. It follows from \eqref{vdecayest2} that
  \begin{equation}\label{vsigmaest}
    \|u_x^\sigma\|_{BC^{k+\mu,(k+\mu)/4}(\R\times(t/2,t))}'=\|u_x\|_{BC^{k+\mu,(k+\mu)/4}(\R\times(\sigma^4t/2,\sigma^4t))}'\le C\|(u_0)_x\|_{W^{1,\infty}}.
  \end{equation}
Furthermore, \eqref{utest} implies that
     \begin{equation}\label{usigmatest}
      \begin{aligned}
        \|u^\sigma_t\|_{BC^{k+\mu,(k+\mu)/4}(\R\times(t/2,t))}'&=\sigma^3\|u_t\|_{BC^{k+\mu,(k+\mu)/4}(\R\times(\sigma t/2,\sigma^4t))}'\\
        &\le C\|(u_0)_x\|_{W^{1,\infty}}(\sigma^4 t)^{-1/2}(1+\sigma^4t)^{-1/4}\\
        &\le C\|(u_0)_x\|_{W^{1,\infty}} t^{-3/4}.
      \end{aligned}
  \end{equation}
  It follows from \eqref{usigmatest} that
  \[
  |u^\sigma(x,t_2)-u^\sigma(x,t_1)|\le C\int_{t_1}^{t_2}t^{-3/4}\,dt\le C(t_2-t_1)^{1/4}
  \]
  for all $x\in\R$ and $0\le t_1<t_2$. This together with \eqref{vsigmaest} implies that
  \begin{gather}
   |u^\sigma(x_2,t_2)-u^\sigma(x_1,t_1)|\le C\left(|x_2-x_1|+|t_2-t_1|^{1/4}\right) \label{Equconti},\\
   |u^\sigma(x,t)|\le |u^\sigma(0,0)|+ C(|x|+t^{1/4})=\sigma^{-1}|u_0(0)|+C(|x|+t^{1/4})\label{Unifbdd},
  \end{gather}
  for $x_1,x_2,x\in\R$, $t_1,t_2,t\in[0,\infty)$.
  
  By the Ascoli--{Arzel\`a} theorem, \eqref{vsigmaest}, \eqref{usigmatest}, \eqref{Equconti}, and \eqref{Unifbdd} imply that there is a sequence $\sigma_n\nearrow\infty$ and $U\in C^{k+1+\mu,(k+1+\mu)/4}_{\textrm{time loc}}(\R\times(0,\infty))\cap C(\R\times[0,\infty))$ such that
  \[
  \lim_{n\to\infty}\|u^{\sigma_n}-U\|_{L^\infty(K)}=0
  \]
  for every compact set $K\subset\R\times[0,\infty)$, and
  \begin{gather*}
  \lim_{n\to\infty}\|\partial_x^{\ell}\partial_t^m(u^{\sigma_n}-U)\|_{L^\infty(K)}=0
  \end{gather*}
  for every compact set $K\subset\R\times(0,\infty)$ and $\ell,m\in\Z_{\ge 0}$ with $\ell+4m\le k+1$. Furthermore, $U$ satisfies the estimate
  \begin{equation}\label{Uxest}
  \|U_x\|_{BC^{k+\mu,(k+\mu)/4}(\R\times(t/2,t))}'\le C\|(u_0)_x\|_{W^{1,\infty}}
  \end{equation}
  for all $t\in(0,\infty)$.
  
  We observe that
  \begin{equation}\label{scaledeq}
  \begin{aligned}
      u^\sigma_t&=\sigma^3\left[\left(\frac{1}{(1+u_x^2)^{1/2}}\left(f\left(-\frac{u_{xx}}{(1+u_x^2)^{3/2}}\right)\right)_x\right)_x\right](\sigma x,\sigma^4t)\\
      &=\sigma\left(\frac{1}{(1+(u^\sigma_x)^2)}\left(f\left(-\frac{\sigma^{-1} u^\sigma_{xx}}{(1+(u^\sigma_x)^2)^{3/2}}\right)\right)_x\right)_x\\
      &=\left(\frac{1}{(1+(u^\sigma_x)^2)^{1/2}}\left(f_\sigma\left(-\frac{u^\sigma_{xx}}{(1+(u^\sigma_x)^2)^{3/2}}\right)\right)_x\right)_x,
   \end{aligned}
    \end{equation}
   where $f_\sigma(r):=\sigma f(\sigma^{-1}r)$.
   Since
   $\lim\limits_{\sigma\to\infty}f_\sigma(r)=r$
   by the assumption \eqref{fassump}, we deduce that
   \[
   U_t=\left(\frac{1}{(1+U_x^2)^{1/2}}\left(-\frac{U_{xx}}{(1+U_x^2)^{3/2}}\right)_x\right).
   \]
   It also follows from \eqref{IVasympt} that
   \[
   U(x,0)=\begin{cases}
     ax\ \textrm{if}\ x\ge 0,\\
     bx\ \textrm{if}\ x<0.
   \end{cases}
   \]
   Hence $U$ is a solution to problem \eqref{SSS}. On the other hand, by \cite{KL}*{Theorem~3.4}, there are constants $\varepsilon^*>0$ and $\rho^*>0$ such that problem \eqref{SSS} possesses a unique solution, which is self-similar, in the set
   $
   \left\{U;\ \|U\|_{X_\infty}<\rho^*\right\}
   $ (see \eqref{KLnorm} for the definition of $\|U\|_{X_\infty}$), 
   provided that $\max\{a,b\}<\varepsilon^*$. These together with \eqref{Uxest} and the fact that
   \[
   \|U\|_{X_\infty}\le C\sup_{t>0}\left(\|U_x\|_{L^\infty}+t^{-1/4}\|U_{xx}\|_{L^\infty}\right)(t)
   \]
   by the H\"{o}lder inequality imply that the limits $U$ of subsequences in $u^\sigma$ is unique. Hence 
   \[
    \lim_{\sigma\to\infty}\|u^\sigma-U\|_{L^\infty(K)}=0
   \]
   for every compact set $\sigma\to\infty$, if $\|(u_0)_x\|_{W^{1,\infty}}$ is sufficiently small. Furthermore,
     \begin{gather*}
\lim_{\sigma\to\infty}\|\partial_x^{\ell}\partial_t^m(u^\sigma-U)\|_{L^\infty(K)}=0
  \end{gather*}
  for every compact set $K\subset\R\times(0,\infty)$ and $\ell,m\in\Z_{\ge 0}$ with $\ell+4m\le k+1$. The proof of Theorem~\ref{ConvSSS} is complete.
\end{proof}

\section*{Acknowledgments}
The work of the first author was partly supported by Japan Society for the Promotion of Science (JSPS) through grants KAKENHI Grant Numbers 19H00639, 20K20342, 24K00531 and 24H00183 and by Arithmer Inc., Daikin Industries, Ltd.\ and Ebara Corporation through collaborative grants. 
 The work of the second author was partially funded by the DFG (Deutsche Forschungsgemeinschaft) within the project 44227998. During this time he was visiting the University of Tokyo in 2022 as JSPS fellow. Both the funding of the DFG and the hospitality of the University of Tokyo are gratefully acknowledged.
 The third author was supported by Grant-in-Aid for JSPS Fellows DC1, Grant Number 23KJ0645.

\end{document}